\documentclass[a4paper,12pt]{amsart}
\DeclareRobustCommand{\SkipTocEntry}[5]{}
\usepackage[english]{babel}
\usepackage[T1]{fontenc}
\usepackage{array}
\usepackage{makecell}
\usepackage[a4paper,margin=3cm]{geometry}
\usepackage[justification=centering]{caption}

\usepackage{amsmath,amsthm,amssymb,amsxtra,calligra,mathrsfs}
\usepackage{graphicx}
\usepackage{tikz}
\usepackage{tikz-cd} 
\usepackage{xcolor}
\usepackage{upgreek}
\usepackage{comment}
\usepackage{graphicx}
\usepackage{mathtools}
\usepackage{extarrows}
\usepackage{faktor}
\usepackage{makecell}
\usepackage{enumitem}
\usepackage{comment}
\usepackage{bookmark}
\usepackage{hyperref}
\usepackage{mathrsfs}
\usepackage{multirow}
\hypersetup{
    colorlinks=true, 
    citecolor=blue,
    linkcolor=magenta,
    pdfauthor={Haoyu Wu},
	bookmarksnumbered=true,
}

\newcommand{\sheafHom}{\mathscr{H}\text{\kern -3pt {\calligra\large om}}\,}

\newcommand*{\rom}[1]{\expandafter\@slowromancap\romannumeral #1@}

\DeclareMathOperator{\vol}{\mathrm{vol}}
\DeclareMathOperator{\ord}{\mathrm{ord}}

\DeclareMathOperator{\val}{Val}
\DeclareMathOperator{\ind}{ind}
\DeclareMathOperator{\lct}{lct}

\DeclareMathOperator{\aut}{Aut}

\DeclareMathOperator{\coeff}{coeff}

\DeclareMathOperator{\codim}{codim}

\newcommand{\q}{/\!\!/}

\newtheorem{defn-pro}{Definition-Proposition}
\newtheorem{defn-thm}{Definition-Theorem}
\newtheorem{thm}{Theorem}[section]
\newtheorem{lem}[thm]{Lemma}

\newtheorem{defn}[thm]{Definition}

\newtheorem{rem}[thm]{Remark}

\newtheorem{notation}{Notation}

\theoremstyle{remark}

\title{Wall crossing for K-moduli space of degree 5 pairs}

\author{Long Pan}
\address{Fudan University}
\email{18110180008@fudan.edu.cn}

\begin{document}

\begin{abstract}
     In this paper, we describe the wall-crossing of the two parameter K-moduli space of pairs $(\mathbb{P}^2,aQ+bL)$, where $Q$ is a plane quintic curve and $L$ is a line.
\end{abstract}

\maketitle
\tableofcontents
\section{Introduction}

Comparison between the different compactifications of the same moduli space is an interesting problem in algebraic geometry. In \cite{laza2007deformationssingularitiesvariationgit}, Laza described the VGIT moduli space $\mathcal{M}(t)$ of the degree 5 pair $(Q,L)$ where $Q$ is a plane quintic curve and $L$ is a line. It can be related to the Baily-Borel compactification of moduli space of certain lattice polarized $K3$ surfaces by viewing the pair $(Q,L)$ as the branching curve of a double cover of $\mathbb{P}^2$. Denote the compactification of the moduli space of degree 5 pairs from above by $(\mathcal{D}/\Gamma)^*$. Moreover, Laza has shown that $\mathcal{M}(1)\cong(\mathcal{D}/\Gamma)^*$.

If we fix any pair of rational numbers $(a,b)\in\left\{(a,b)~|~\text{$5a+b<3$, $a\ge0$, $b\ge0$}\right\}$, then $(\mathbb{P}^2,aQ+bL)$ is a log Fano pair, and the theory of K-stability provides a natural framework to construct alternative compactifications of the moduli space of degree 5 pairs. This framework is established in \cite{ADL19} and \cite{zhou2023chamberdecompositionksemistabledomains}. Let $\overline{\mathcal{P}}^K_{(a,b)}$ be the moduli stack parametrizing $K$-semistable log Fano pairs which admit $\mathbb{Q}$-Gorenstein smoothings to $(\mathbb{P}^2,aQ+bL)$, where $(Q,L)$ is a degree 5 pair. The goal of this paper is to describe the wall-crossing of this moduli space. 
Firstly, we find out all possible K-semistable degenerations of $(\mathbb{P}^2,aQ+bL)$:

\begin{thm}
    If K-semistable log Fano pair $(X,aC+bL)$ is parameterized by $\overline{\mathcal{P}}^K_{(a,b)}$, then $X$ can only possibly be $\mathbb{P}^2$, $\mathbb{P}(1,1,4)$, $\mathbb{P}(1,4,25)$ or $X_{26}$.
\end{thm}
Then according to above classification, we can find all strictly K-polystable log Fano pairs parameterized by $\overline{\mathcal{P}}^K_{(a,b)}$ and their K-stability threshold:

\begin{thm}
The critical log Fano pairs and the corresponding critical lines are listed in the following table:

\begin{center}
    \renewcommand*{\arraystretch}{1.2}
    \begin{table}[ht]
    \centering
      \begin{tabular}{|c |c |c |c|}
    \hline
     surface &wall &  curve $Q$ on $X$ & curve $L$ on $X$  \\  \hline
     
      \multirow{14}{*}{$\mathbb{P}^2$} &$2b=5a$   &  $\left\{zxy(x+\lambda_1y)(y+\lambda_2x)=0\right\}$   & $\left\{z=0\right\}$ \\ \cline{2-4} 

      &$5b=11a$  &  $\left\{x^2(y^3+x^2z)=0\right\}$   & $\left\{z=0\right\}$
      \\ \cline{2-4}   
      
      &$b=2a$ &     $\left\{yx^2(zx+y^2)=0\right\}$   & $\left\{z=0\right\}$ 
      \\ \cline{2-4}  
      
      &$7b=13a$&     $\left\{xy(y^3+zx^2)=0\right\}$   & $\left\{z=0\right\}$ 
      \\ \cline{2-4}
      
      &$4b=7a$ &     $\left\{x^2(y^3+z^2x)=0\right\}$   & $\left\{z=0\right\}$ 
      \\ \cline{2-4}
      
      &$3b=5a$ &     $\left\{y^5+x^4z=0\right\}$   & $\left\{z=0\right\}$ 
      \\ \cline{2-4}
      
      &$5b=8a$ &     $\left\{xy(y^3+x^2z)=0\right\}$   & $\left\{z=0\right\}$
      \\ \cline{2-4}
          
      &$7b=10a$ &     $\left\{y(y^4+x^3z)=0\right\}$   & $\left\{z=0\right\}$ \\ \cline{2-4}

     &$a=b$ &  $\left\{y^2(y^3+x^2z)\right\}$  & $\left\{z=0\right\}$ 
     \\ \cline{2-4}
     
     &$8b=5a$ &     $\left\{y^5+x^3z^2=0\right\}$   & $\left\{z=0\right\}$ 
     \\ \cline{2-4}
     &$5b=2a$ &     $\left\{xz(y^3+x^2z)=0\right\}$   & $\left\{z=0\right\}$ 
     \\ \cline{2-4}
     &$4b=a$ &     $\left\{yx(y^3+z^2x)=0\right\}$   & $\left\{z=0\right\}$ 
     \\ \cline{2-4}
     &$7b=a$ &     $\left\{z(y^4+x^3z)=0\right\}$   & $\left\{z=0\right\}$ 
     \\ \hline
     
     \multirow{6}{*}{$\mathbb{P}(1,1,4)$} 
     &$7a-b=3$   &  $\left\{z^2xy=0\right\}$  & $\left\{xy=0\right\}$ or $l_1l_2=0$
     \\ \cline{2-4}
     & $35a-17b=15$  &  $\left\{z^2xy+y^{10}=0\right\}$   & $\left\{x^2=0\right\}$
     \\ \cline{2-4}
     &$7a-4b=3$ &     $\left\{z^2xy+zx^5y=0\right\}$   & $\left\{y^2=0\right\}$ 
     \\ \cline{2-4}
     & $11a+b=6$ &     $\left\{x^2z^2+y^6z=0\right\}$   & $\left\{y^2=0\right\}$ 
     \\ \cline{2-4}
     & $11a-2b=6$ &     $\left\{x^2z^2+y^6z=0\right\}$   & $\left\{xy=0\right\}$ 
     \\ \cline{2-4}
     & $11a-5b=6$ &     $\left\{x^2z^2+y^6z=0\right\}$   & $\left\{x^2=0\right\}$ 
     \\ \hline 
     
     \multirow{6}{*}{$\mathbb{P}(1,4,25)$}&$115a+11b=63$   &  $\left\{z^2+x^2y^{12}=0\right\}$   & $\left\{x^2y^2=0\right\}$ 
     \\ \cline{2-4} 
     &$115a-19b=63$  &  $\left\{z^2+x^2y^{12}=0\right\}$   & $\left\{x^6y=0\right\}$
     \\ \cline{2-4}   
     &$115a-49b=63$ &     $\left\{z^2+x^2y^{12}=0\right\}$   & $\left\{x^{10}=0\right\}$ 
     \\ \cline{2-4}  
     &$95a+13b=54$ &     $\left\{z^2+x^6y^{11}=0\right\}$   & $\left\{x^2y^2=0\right\}$ 
     \\ \cline{2-4}
     &$95a-17b=54$ &     $\left\{z^2+x^6y^{11}=0\right\}$   & $\left\{x^6y=0\right\}$ 
     \\ \cline{2-4}
     &$95a-47b=54$ &     $\left\{z^2+x^6y^{11}=0\right\}$   & $\left\{x^{10}=0\right\}$ 
     \\ \hline
     
     \multirow{3}{*}{$X_{26}$}&$45a-17b=24$   &  $\left\{w=0\right\}$   & $\left\{x^5=0\right\}$ 
      \\ \cline{2-4}  
     &$45a-7b=24$  &  $\left\{w=0\right\}$   & $\left\{x^3y=0\right\}$\\ \cline{2-4}  
     &$15a+b=8$ &     $\left\{w=0\right\}$   & $\left\{xy^2=0\right\}$ \\ \hline
     \end{tabular}
     \caption{ Critical pairs and critical lines}
     \label{Kwall2}
\end{table}
\end{center}
where $l_2$ and $l_2$ is degree 1 form of $x,y$ and $l_i\neq x,y$.
\end{thm}

Finally, we will describe the wall crossing morphism of the K-moduli space $\overline{P}^K_{a,b}$. We show that:
\begin{thm}(=Theorem \ref{horizontalwall})
    Let $(a,b)$ be a pair of positive rational numbers $(a,b)$ such that $5a+b\le3$, $2b\le5a$ and $7a-b\le3$. Then we have the isomorphism between the GIT-moduli space of degree 5 pairs and the K-moduli space: $\overline{\mathcal{M}}(t)\cong\overline{P}_{a,b}^K$, where $t=\dfrac{b}{a}$.
\end{thm}
Then we describe the wall crossing morphisms at the walls $7a-b=3$ and $15a+b=8$. These two walls are divisorial contractions and the rest walls are both flips. Moreover, they are also the walls where the surfaces $\mathbb{P}(1,1,4)$ and $X_{26}$ 
first appears in the K-moduli space $\overline{P}^K_{a,b}$ respectively:

\begin{thm}(=Theorem \ref{type1wallcrossing})
    Let $(a_1,b_1)$ be a pair of positive rational numbers such that $7a_1-b_1=3$ and $5a_1+b_1\le3$. Fix any another two pairs of rational numbers $(a_{1}^{-},b_{1}^{-})$ and $(a_{1}^{+},b_{1}^{+})$ such that $7a_{1}^{-}-b_{1}^{-}<3$ and $7a_{1}^{+}-b_{1}^{+}>3$. Moreover ,we assume that they are lying in the two distinct chambers which has a common face along the the critical line $7a-b=3$. 
    \begin{enumerate}
        \item The general point of the center $\Sigma_{a_1,b_1}$ parameterizes the pair $(\mathbb{P}(1,1,4);Q^{(1)},L^{(1)})$, where $L^{(1)}$ is defined by $\left\{xy=0\right\}$ or $\left\{\text{$l_1l_2=0$ | $l_i\neq x$ and $y$}\right\}$. In particular, $\dim\Sigma_{a_1,b_1}=1$.
        \item The wall crossing morphism $\phi^{-}_{a_1,b_1}:\overline{P}^K_{a_{1}^{-},b_{1}^{-}}\rightarrow\overline{P}^K_{a_1,b_1}$ is isomorphism which take the points $[(\mathbb{P};Q_{0},L)]$ to $[(\mathbb{P}(1,1,4);Q^{(1)},\{xy=0\})]$ and takes the point $[(\mathbb{P};Q_{0},L')]$ to $[(\mathbb{P}(1,1,4);Q^{(1)},\{L^{(2)'}=0\})]$.
        \item The wall crossing morphism $\phi^{+}_{a_1,b_1}:\overline{P}^K_{a_1^+,b_1^+}\rightarrow\overline{P}^K_{a_1,b_1}$ is a divisorial contraction. The exceptional divisor $E^{+}_{a_1,b_1}$ parameterizes pair $(\mathbb{P}(1,1,4);Q,L)$ such that $(Q,L)$ is GIT-polystable with slope $(a_1^+,b_1^+)$ in the sense of definition \ref{GITonP114}.
    \end{enumerate}
\end{thm}

\begin{thm}(=Theorem \ref{type2wallcrossing})
    Let $(a_2,b_2)$ be a pair of positive rational numbers such that $15a_2+b_2=8$ and $5a_1+b_1\le3$. Fix any another two pairs of rational numbers $(a_{1}^{-},b_{1}^{-})$ and $(a_{1}^{+},b_{1}^{+})$ such that $15a_{1}^{-}+b_{1}^{-}<8$ and $15a_{1}^{+}+b_{1}^{+}>8$. Moreover ,we assume that they are lying in the two distinct chambers which has common face along the the critical line $15a+b=8$.
    \begin{enumerate}
        \item The wall crossing morphism $\phi^{-}_{a_2,b_2}:\overline{P}^K_{a_{2}^{-},b_{2}^{-}}\rightarrow\overline{P}^K_{a_2,b_2}$ is isomorphism which take the points $[(\mathbb{P};Q_{2},L)]$ to $[(X_{26};Q^{(2)},\{xy^2=0\})]$.
        \item The wall crossing morphism $\phi^{+}_{a_2,b_2}:\overline{P}^K_{a_{2}^{+},b_{2}^{+}}\rightarrow\overline{P}^K_{a_2,b_2}$ is a divisorial contraction. The exceptional divisor $E^{+}_{a_2,b_2}$ parametrizes the pair of curves on $X_{26}$ of the form $(\left\{w=g(x,y)\right\},xy^2=h(x,y))$ where $g\neq0$ and $g$ does not contain the term $xy^{12}$, $h$ is degree $5$ of the form $\left\{\lambda x^5+\mu x^3y\right\}$.
    \end{enumerate}
\end{thm}

\subsection*{Structure of the paper}
In section 2 we will recall some basic concepts and important results in the theory of $K$-stability;
In section 3 we will give a brief review about normalized volume and $T$-singularity which enable us determine the K-semistable degenerations in the moduli space we consider;
Then in section 4 we will use the technique about complexity one varieties to figure out all critical pairs and critical lines, which describe the wall-chamber structure of the moduli space $\overline{P}^K_{(a,b)}$.
Finally, in section 5 we will describe the wall crossing morphisms which are divisorial contractions.

\subsection*{Acknowledgments}  The author is grateful to Yuchen Liu for introducing this problem to him. The authors also benefit a lot from the discussion with Zhiyuan Li and Fei Si. He also thanks Chuyu Zhou, Xun Zhang and Haitao Zou for reading the drafts.

\section{Preliminaries}

\begin{defn}
    Let X be a projective variety. Let $L$ be an ample line bundle on $X$. A test configuration $(\mathcal{X};\mathcal{L})/\mathbb{A}^1$ of $(X;L)$ consists of the following data:
     \begin{itemize}
         \item a variety $\mathcal{X}$ together with a flat projective morphism $\pi:\mathcal{X}\longrightarrow\mathbb{A}^1$;
         \item a $\pi$-ample line bundle $\mathcal{L}$ on $\mathcal{X}$;
         \item a $\mathbb{G}_m$-action on $(\mathcal{X;L})$ such that $\pi$ is $\mathbb{G}_m$-equivariant with respect to the standard action of $\mathbb{G}_m$ on $\mathbb{G}_m$ on $\mathbb{A}^1$ via multiplication; 
         \item $(\mathcal{X}\setminus \mathcal{X}_0, \mathcal{L}|_{\mathcal{X}\setminus \mathcal{X}_0})$ is $\mathbb{G}_m$-equivariantly isomorphic to $(X;L)\times(\mathbb{A}^1\setminus \{0\})$.
     \end{itemize}
     If $(X, D=\sum^k_{i=1}c_iD_i)$ is a projective log pair, then the test configuration $(\mathcal{X},\mathcal{D};\mathcal{L})$ of $(X,D;L)$ should adding the following data:
      \begin{itemize}
          \item a form sum $\mathcal{D}=\sum^k_{i=1}c_i\mathcal{D}_i$ of codimension one closed integral subschemes $\mathcal{D}_i$ of $\mathcal{X}$ such that $\mathcal{D}_i$ is the Zariski closure of $D_i\times\mathbb{A}^1\setminus\{0\}$ under the identification between $\mathcal{X}\setminus\mathcal{X}_0$ and $X\times(\mathbb{A}^1\setminus\{0\})$.
      \end{itemize}
\end{defn}

\begin{defn}
   Notation as above.
      \begin{enumerate}
       \item A test configuration $(\mathcal{X},\mathcal{L})/\mathbb{A}^1$ is called a normal test configuration if $\mathcal{X}$ is normal; 
       \item A normal test configuration is called a product test configuration if
       \begin{equation*}
        (\mathcal{X};\mathcal{L})\cong(X\times\mathbb{A}^1,pr^*_1L\otimes\mathcal{O}_{\mathcal{X}}(k\mathcal{X}_0))
        \end{equation*}
       for some $k\in\mathbb{Z}$; 
       \item A product test configuration is called a trivial test figuration if the above isomorphism is $\mathbb{G}_m$-equivariant with respect to the trivial $\mathbb{G}_m$-action on $X$ and the standard $\mathbb{G}_m$-action on $\mathbb{A}^1$ via multiplication.
        \end{enumerate}
\end{defn}

\begin{defn}
    Let $w_k$ be the total weight of the $\mathbb{G}_m$-action on the determinant line $det~H^0(X_0,L_0^{\otimes m})$. By the equivariant Riemann-Roch Theorem,  
\begin{equation*}
    w_k=b_0k^{n+1}+b_1k^n+O(k^{n-1}).
\end{equation*}
 For a sufficiently divisible $k\in\mathbb{N}$, we have
\begin{equation*}
    N_k:=dimH^0(X,\mathcal{O}_X(-kK_X))=a_0k^{n}+a_1k^{n-1}+O(k^{n-2}).
\end{equation*}
So we can expand
\begin{equation*}
    \frac{w_k}{kN_k}=F_0+F_1k^{-1}+O(k^{-2})
\end{equation*}
Under the above notation, the generalized Futaki invariant of the test configuration $(\mathcal{X};\mathcal{L})$ is defined to be 
        \begin{equation*}
        Fut(\mathcal{X};\mathcal{L})=-F_1=\frac{a_1b_0-a_0b_1}{a_0^2}
        \end{equation*}

    For projective log pair $(X,D=\sum^k_{i=1}c_iD_i)$, we should add the contribution from the boundary:  we denote the relative Chow weight of $(\mathcal{X},\mathcal{D};\mathcal{L})$ by $CH(\mathcal{X},\mathcal{D};\mathcal{L}) $ and the generalized Futaki invariant of $(\mathcal{X},\mathcal{D};\mathcal{L})$ is defined as 
        \begin{equation*}
Fut(\mathcal{X},\mathcal{D};\mathcal{L})=Fut(\mathcal{X},\mathcal{L})+CH(\mathcal{X},\mathcal{D};\mathcal{L}).
        \end{equation*}
\end{defn}

\begin{defn}
    Let $(X,D)$ be a log Fano pair. Let $L$ be an ample line bundle on $X$ such that for some $r\in\mathbb{Q}_{>0}$ such that $L\sim_{\mathbb{Q}_{>0}}-r(K_X+D)$, Then the log Fano pair $(X,D)$ is said to be:
    \begin{itemize}
        \item K-semistable if $Fut(\mathcal{X},\mathcal{D};\mathcal{L})\ge0$ for any normal test configuration $(\mathcal{X},\mathcal{D};\mathcal{L})$ and any $r\in\mathbb{Q}_{>0}$ such that $L$ is Cartier;
        \item K-polystable
        if it is K-semistable and $Fut(\mathcal{X},\mathcal{D};\mathcal{L})=0$ for a normal test configuration $(\mathcal{X},\mathcal{D};\mathcal{L})$ if and only if it is a product test configuration; 
        \item K-stable if it is K-semistable and $Fut(\mathcal{X},\mathcal{D};\mathcal{L})=0$ for a normal test configuration $(\mathcal{X},\mathcal{D};\mathcal{L})$ if and only if it is a trivial test configuration.
        \end{itemize}
\end{defn}

Then we recall the valuative criterion of K-stability.

Let $(X,D)$ be a log Fano pair. We call $E$ is a prime divisor over $(X,D)$ if it contained in the birational model of $X$, i.e. there exists a normal projective variety $Y$ and a projective birational morphism $f:Y\rightarrow X$, such that $E\subset Y$.

\begin{defn}
    Notation as above. Then we can associate two functional to $E$ as follows:
    \begin{itemize}
        \item $A_{(X,D)}(E):=1+\ord_{E}(K_{Y}-f^*(K_X+D))$;
        \item $S_{(X,D)}(E):=\dfrac{1}{\vol(-K_X-D)}\int^{+\infty}_{0}\vol(-f^*(K_X+D)-tE)dt$
    \end{itemize}
    Let $\beta_{(X,D)}(E):=A_{(X,D)}(E)-S_{(X,D)}(E)$.
\end{defn}

\begin{thm}(\cite{Fujita},\cite{Li17},\cite{BLX19})
    Let $(X,D)$ be a log Fano pair. Then the following holds:
    \begin{itemize}
        \item $(X,D)$ is K-semistable$\Longleftrightarrow$ $\beta_{(X,D)}(E)\ge0$ for any prime divisor $E$ over $X$;
        \item $(X,D)$ is K-stable$\Longleftrightarrow$ $\beta_{(X,D)}(E)>0$ for any prime divisor $E$ over $X$;
    \end{itemize}
\end{thm}

When we consider the Fano variety with some torus action, it can be easier to check the K-stability of these pairs by means of the torus invariant divisors.

\begin{defn}
    Let $(X,D)$ be a log Fano pair. Let $\mathbb{T}\subset\aut(X,D)$ be a maximal torus. We call $\dim X-\dim T$ the complexity of the $\mathbb{T}$-action.
\end{defn}

\begin{defn}
    Let $E$ be a $\mathbb{T}$-invariant divisor over $X$. It is said to be vertical if a maximal $\mathbb{T}$-orbit in $E$ has the same dimension as the torus $\mathbb{T}$. Otherwise, the divisor $E$ is said to be horizontal.
\end{defn}

Given a log Fano pair $(X,D)$ with a maximal torus $\mathbb{T}$ action. Let $L=\mathcal{O}_{X}(-lK_X)$, here $l\gg0$ is an integer. For every one parameter subgroup $\lambda$ on $X$ and its canonical linearization for $L$, we set
\[
w_k(\lambda)=\sum_{m}m\cdot\dim(H^0(X,L^{\otimes k})_m),
\]
where $H^0(X,L^{\otimes k})_{m}$ is the subspace of the semi-invariant sections of $\lambda$-weight $m$.

\begin{defn}
    The function
    \[
    Fut_{(X,D)}(\lambda):=-\lim_{k\rightarrow\infty}\frac{w_{k}(\lambda)}{k\cdot l_k\cdot l},
    \]
    is called the Futaki character of the log Fano pair $(X,D)$.
\end{defn}

\begin{thm}\label{complexity1}(\cite{Calabi})
    Suppose $(X,D)$ is a log Fano pair, the complexity of the $\mathbb{T}$-action is 1. Then $(X,D)$ is K-polystable if and only if the following conditions hold:
    \begin{itemize}
        \item $\beta_{(X,D)}(E)>0$ for every vertical $\mathbb{T}$-invariant prime divisor $E$ on $X$;
        \item $\beta_{(X,D)}(E)=0$ for every horizontal $\mathbb{T}$-invariant prime divisor $E$ on $X$;
        \item $Fut_{(X,D)}(\lambda)=0$ for any one parameter subgroup $\lambda$ of $\mathbb{T}$.
    \end{itemize}
\end{thm}

The K-stability enable us to construct the compact moduli space of Fano varieties or log Fano pairs. More precisely, It has been proven that:

\begin{thm}[\cite{ADL19}]
    Let $\chi_0$ be the Hilbert polynomial of an anti-canonically polarized Fano manifold. Fix $r\in\mathbb{Q}_{>0}$ and a rational number $c\in(0,\min\left\{1,r^{-1}\right\})$. Consider the following moduli pseudo-functor over reduced base $S$:
    \[
    \mathcal{KM}_{\chi_0,r,c}(S)=\left\{
    (\mathcal{X},\mathcal{D})/S\Bigg|
    \begin{array}{cc} 
    (\mathcal{X},\mathcal{D})/S\text{ is a $\mathbb{Q}$-Gorenstein smoothable log Fano family,}& \\
    \mathcal{D}\thicksim_{S,\mathbb{Q}}, \text{each fiber $(\mathcal{X}_s,c\mathcal{D}_s)$ is K-semistable,}& \\
    \text{and $\chi(\mathcal{X}_s,\mathcal{O}_{\mathcal{X}_s}(-kK_{\chi_s}))=\chi_0(k)$ for $k$ sufficiently divisible.} &
    \end{array}\right\}
    \]
    Then there exists a reduced Artin stack $\mathcal{KM}_{\chi_0,r,c}$ (called a K-moduli stack) of finite type over $\mathbb{C}$ representing the above moduli pseudo-functor. In particular, the $\mathbb{C}$-points of $\mathcal{KM}_{\chi_0,r,c}$ parameterize K-semistable $\mathbb{Q}$-Gorenstein smoothable log Fano pairs $(X,cD)$ with Hilbert polynomial $\chi(X,\mathcal{O}_{X}(-mK_X))=\chi_0(m)$ for sufficiently divisible $m$ and $D\sim_{\mathbb{Q}}-rK_{X}$.

    Moreover, the Artin stack $\mathcal{KM}_{\chi_0,r,c}$ admits a good moduli space $KM_{\chi_0,r,c}$ (called K-moduli space) as a proper reduced scheme of finite type over $\mathbb{C}$, whose closed points parameterize K-polystable log Fano pairs.
\end{thm}

\begin{rem}
    It has been shown that the $CM$-line bundle $\Lambda_c$ over $KM_{\chi_0,r,c}$ is ample. Hence $KM_{\chi_0,r,c}$ is a projective scheme. Furthermore, if $(\mathcal{X},\mathcal{D})/S$ is family of log surface pairs, then both $\mathcal{KM}_{\chi_0,r,c}$ and $KM_{\chi_0,r,c}$ are normal.
\end{rem}

It is thus natural to ask how moduli spaces depend on the coefficient $c$. In the above setting, they prove the following wall-crossing type result.

\begin{thm}(\cite{ADL19})
    There exist rational numbers
    \[
    0=c_0<c_1<c_2<...<c_k=\min\left\{1,r^{-1}\right\}
    \]
    such that $c$-$k$-(poly/semi)stability conditions do not change for $c\in(c_i,c_{i+1})$. For each $1\le i\le k-1$, we have open immersions
    \[
    \mathcal{KM}_{\chi_0,r,c}\stackrel{\Phi^{-}_{i}}{\hookrightarrow}\mathcal{KM}_{\chi_0,r,c-\epsilon}\stackrel{\Phi^{+}_{i}}{\hookleftarrow}\mathcal{KM}_{\chi_0,r,c+\epsilon}
    \]
    which induce projective morphisms
    \[
    KM_{\chi_0,r,c}\stackrel{\phi^{-}_{i}}{\rightarrow}KM_{\chi_0,r,c-\epsilon}\stackrel{\phi^{+}_{i}}{\leftarrow}KM_{\chi_0,r,c+\epsilon}
    \]
    Moreover, all the above wall crossing morphisms have local VGIT presentations, and the $CM$ $\mathbb{Q}$-line bundles on $KM_{\chi_0,r,c\pm\epsilon}$ are $\phi^{\pm}_{i}$-ample.
\end{thm}

Similarly, if one consider the log Fano pair with multiple boundaries, there exists a similarly result as above:

\begin{thm}(\cite{zhou2023shapeksemistabledomainwall})
    Fix two positive integers $d$ and $k$, a positive number $v$, and a finite set $I$ of non-negative rational numbers. There exists a finite chamber decomposition $\overline{\Delta^k}=\bigcup^{r}_{i=1}P_i$, where $P_i$ are rational polytopes and $P_{i}^{\circ}\cap P_{j}^{\circ}=\emptyset$ for $i\neq j$, such that
    \begin{enumerate}
        \item $\mathcal{M}^{K}_{d,k,v,I,\overrightarrow{c}}$ (resp. $M^K_{d,k,v,I,\overrightarrow{c}}$) does not change as $\overrightarrow{c}$ varies in $P^{\circ}_{i}(\mathbb{Q})$;
        \item Let $P_i$ and $P_j$ be two different polytopes in the decomposition that share the same face $F_{i,j}$. Suppose $\overrightarrow{w}_1\in P^{\circ}_{i}(\mathbb{Q})$ and $\overrightarrow{w}_2\in P^{\circ}_{j}(\mathbb{Q})$ satisfy that the segment connecting them intersects $F_{ij}$ at a point $\overrightarrow{w}$, then $\overrightarrow{w}$ is a rational point and for any rational numbers $0<t_1,t_2<1$, $\mathcal{M}^{K}_{d,k,v,I,t_1\overrightarrow{w}+(1-t_1)\overrightarrow{w}_1}$ (resp. $\mathcal{M}^{K}_{d,k,v,I,t_2\overrightarrow{w}+(1-t_2)\overrightarrow{w}_2}$) does not change as $t_l$ varies in $[0,1)\cap\mathbb{Q}$,where $l=1,2$. 
        \item For any $t_1$, $t_2\in[0,1)\cap\mathbb{Q}$, there exist open immersions:
        \[    \mathcal{M}^{K}_{d,k,v,I,t_1\overrightarrow{w}+(1-t_1)\overrightarrow{w}_1}\stackrel{\Phi^{-}_{i}}{\hookrightarrow}\mathcal{M}^{K}_{d,k,v,I,\overrightarrow{w}}\stackrel{\Phi^{-}_{i}}{\hookleftarrow}\mathcal{M}^{K}_{d,k,v,I,t_2\overrightarrow{w}+(1-t_2)\overrightarrow{w}_2}.
        \]
        which induce projective morphisms
        \[
        M^{K}_{d,k,v,I,t_1\overrightarrow{w}+(1-t_1)\overrightarrow{w}_1}\stackrel{\Phi^{-}_{i}}{\rightarrow}M^{K}_{d,k,v,I,\overrightarrow{w}}\stackrel{\Phi^{-}_{i}}{\leftarrow}M^{K}_{d,k,v,I,t_2\overrightarrow{w}+(1-t_2)\overrightarrow{w}_2}.
        \]
        
    \end{enumerate}
    
\end{thm}

\section{K-semistable degeneration}
In this section, we will describe the K-semistable degeneration of degree 5 pair. Firstly, let's recall the definition and basic properties of $T$-singularities.

\begin{defn}
    Let $(p\in X)$ be a quotient singularity of dimension 2. We say $(p\in X)$ is a T-singularity if it admits a $\mathbb{Q}$-Gorenstein smoothing. That is, there exists a deformation of $(p\in X)$ over the germ of a curve such that the total space is $\mathbb{Q}$-Gorenstein and the general fibre is smooth.
\end{defn}

For $n$, $a_1,...,a_n\in\mathbb{N}$, let $\frac{1}{n}(a_1,...a_n)$ denote the cyclic quotient singularity $(0\in\mathbb{A}^n/\mu_n)$ given by 
\[
\mu_n\ni\zeta:(x_1,...,x_n)\mapsto (\zeta^{a_1},...,\zeta^{a_n}).
\]
Here $\mu_n$ denotes the group of $n$-th roots of unity.
\begin{thm}(\cite{KSB})
     A T-singularity is either a Du Val singularity or a cyclic quotient singularity of the form $\frac{1}{dn^2}(1,dna-1)$ for some $d$, $n$, $a\in\mathbb{N}$ with $(a,b)=1$.
\end{thm}

\begin{notation}
    For $p\in X$ a T-singularity, let $\mu_p$ denote the Milnor number of a $\mathbb{Q}$-Gorenstein smoothing of $(p\in X)$.
\end{notation}

If $(p\in X)$ is Du Val singularity of type $A_n$, $D_n$ or $E_n$ then $\mu_p=n$. If $(p\in X)$ is of type $\dfrac{1}{dn^2}(1,dna-1)$ then $\mu_p=d-1$.

\begin{thm}(\cite{hacking2005degenerationsdelpezzosurfaces})
    Let $X$ be a rational surface with $T$-singularities. Then 
    \[
    K_X^2+\rho(X)+\sum_{p}\mu_p=10.
    \]
    where the sum is over the singular points $p\in X$.
\end{thm}

Then we give a brief review of normalized volume of valuations, which enable us to estimate the Gorenstein index around singularities.

\begin{defn}
    Let $(X,D)$ be a klt pair of dimension $n$. Let $x\in X$ be a closed point. A valuation $v$ on $X$ centered at $x$ is a valuation of $\mathbb{C}(X)$ such that $v\ge0$ on $\mathcal{O}_{X,x}$ and $v>0$ on $m_{x}$. The set of such valuations is denoted by $\val_{X,x}$. The volume is a function $\vol_{X,x}: Val_{X,x}\rightarrow\mathbb{R}_{\ge0}$ defined as
    \[
    \vol_{X,x}(v):=\lim_{k\rightarrow\infty}\frac{\dim_{\mathbb{C}}\mathcal{O}_{X,x}/\left\{
    f\in\mathcal{O}_{X,x}|v(f)\ge k\right\}
    }{k^n/n!}.
    \]
    The log discrepancy is a function $A_{(X,D)}:\val_{X,x}\rightarrow \mathbb{R}_{\ge0}\cup\left\{+\infty\right\}$.
    If $v=a\cdot\ord_{E}$ where $a\in\mathbb{R}_{>0}$ and $E$ is a prime divisor over $X$ centered at $x$, then 
    \[
    A_{(X,D)}(v)=a(1+\ord_{E}(K_{Y/X})),
    \]
    where $Y$ is a birational model of $X$ containing $E$ as a divisor.

    The normalized volume is a function $\widehat{\vol}_{(X,D),x}:\val_{X,x}\rightarrow\mathbb{R}_{>0}\cup\left\{+\infty\right\}$ defined as 
    \[
    \widehat{\vol}_{(X,D),x}(v):=
    \left\{
    \begin{array}{cc}
        A_{(X,D)}(v)^n\cdot\vol_{X,x}(v), & if\quad A_{(X,D)}(v)<+\infty  \\ 
         +\infty & if\quad A_{(X,D)}(v)=+\infty 
    \end{array}
    \right.
    \]
    The local volume of a klt singularity $x\in(X,D)$ is defined as 
    \[
    \widehat{\vol}(x,X):=\min_{v\in\val_{X,x}}\widehat{\vol}_{(X,D),x}(v).
    \]
\end{defn}

\begin{thm}\cite{LL19}
    Let $(X,D)$ be a K-semistable log Fano pair of dimension n. Then for any closed point $x\in X$, we have 
    \[
    (-K_X-D)^n\le(1+\frac{1}{n})^n\widehat{\vol}(x,X,D).
    \]
\end{thm}

Now we can prove the following them about the K-semistable degeneration of log Fano pair:

\begin{lem}\label{index}
    Let $(X,c_1D_1+...+c_mD_m)$ be K-semistable log Fano pair that admits a $\mathbb{Q}$-Gorenstein smoothing to $(\mathbb{P}^2,c_1C_1+...+c_mC_m)$ with $\sum^{m}_{i=1} c_id_i\in(0,3)$. Let $x\in X$ be any singular point with Gorenstein index $\ind(x,X)$, then
    \[
    ind(x,K_X)\le
    \left\{
    \begin{array}{cc}
        \min\left\{\lfloor\dfrac{3}{3-\sum^{m}_{i=1} c_id_i}\rfloor,\quad \max\left\{d_i\right\}_{i=1}^{m}\right\} &  if\quad3\nmid d_i\quad for\quad all\quad i,\\
        \min\left\{\lfloor\dfrac{3}{3-\sum^{m}_{i=1} c_id_i}\rfloor,\quad \max\left\{\dfrac{2d_i}{3}\right\}_{i=1}^{m}\right\} & if\quad3\mid d_i\quad for\quad some\quad i.
    \end{array}
    \right.
    \]
\end{lem}
\begin{proof}
    The proof is the same as in \cite{ADL19}.
\end{proof}

\section{K-stability of complexity one pairs}

In this section, we will compute the K-stability threshold of the complexity one log Fano pairs. We will see that such pair will appear in the wall of K-moduli space as the centers of wall crossing morphisms.

\begin{defn}
    Fix some $(a,b)\in\left\{\text{$(x,y)$ $|$ $5x+y\le3$, $x\in\mathbb{Q}_{\ge0}$, $y\in\mathbb{Q}_{\ge0}$}\right\}$. We call this region by log Fano domain.
    Let's denote the moduli stack which parameterize the log Fano pair that admits a $\mathbb{Q}$-Gorenstein smoothing to $(\mathbb{P}^2,aQ+bL)$ by $\overline{\mathcal{P}}^K_{(a,b)}$ and denote the corresponding moduli space by $\overline{P}^K_{(a,b)}$, where $Q\in|\mathcal{O}_{\mathbb{P}^2}(5)|$ is a degree 5 plane curve and $L\in|\mathcal{O}_{\mathbb{P}^2}(1)|$ is a line.
\end{defn}

Firstly, Let's figure out the total space of K-semistable log Fano pairs contained in the K-moduli space. This is the direct consequence of the index estimate:

\begin{thm}\label{degeneration}
    If K-semistable log Fano pair $(X,aQ+bL)$ is parameterized by $\overline{\mathcal{P}}^K_{(a,b)}$, then $X$ is only possible isomorphic to $\mathbb{P}^2$, $\mathbb{P}(1,1,4)$, $\mathbb{P}(1,4,25)$ or $X_{26}$.
\end{thm}

\begin{proof}
    By lemma\ref{index}, we know that the Gorenstein index of $X$ at any singular points is no bigger than $5$. Thus by the classification of log del Pezzo surfaces \cite{Hacking}, we conclude that X could only possible be $\mathbb{P}^2$, $\mathbb{P}(1,1,4)$, $\mathbb{P}(1,4,25)$ or $X_{26}$.
\end{proof}

Before we compute the walls of K-moduli space, we should notice the following easy but useful result:

\begin{thm}\label{critical}
    Notation as above. There exists a wall-crossing morphisms at $(a_0,b_0)$    
    if and only if there exists complexity one or toric $(a_0,b_0)$-K-polystable log Fano pairs. We call these pairs the "critical" pairs.
\end{thm}

\begin{proof}
    If the moduli space admits a wall crossing:
    \[
    \begin{tikzcd}
    \overline{P}^K_{(a_0-\epsilon,b_0-\delta)} \arrow[rd]\arrow[rr,dashrightarrow] & & \overline{P}^K_{(a_0+\epsilon,b_0+\delta)} \arrow[ld]\\
    & \overline{P}^K_{(a_0,b_0)} &
    \end{tikzcd}
    \]
    there must be some strictly K-polystable pairs $(X,a_0Q+b_0L)$ in $\overline{P}^K_{(a_0,b_0)}$, hence the automorphism group $Aut(X;Q,L)$ is reductive group (\cite{JHHLX}). Thus the maximal torus of $\aut{(X,a_0Q+b_0L)}$ have positive dimension.
\end{proof}

Thus in order to look for the K-moduli walls, we only need to compute the K-stability threshold of critical pairs. By theorem \ref{degeneration}, let's figure out these pairs case by case by case. 

\begin{notation}
    Let $(X,aQ+bL)$ be a K-semistable log Fano pair, where $X$ equals to $\mathbb{P}^2$, $\mathbb{P}(1,1,4)$, $\mathbb{P}(1,4,25)$ or $X_{26}$ and $C\sim-\frac{5}{3}K_{X}$, $L\sim-\frac{1}{3}K_{X}$. Let $\aut(X;Q,L)$ be the subgroup of $\aut(X)$ which is generated by the automorphism preserved $Q$ and $L$ simultaneously. Denote the maximal torus by $\mathbb{T}\subset\aut(X;Q,L)$. When we consider the prime divisor $E$ over $X$, we assume the birational model contains $E$ by $\mu: Y\rightarrow X$.
\end{notation}

\subsection{The critical $\mathbb{P}^2$-pair}

In this case, $X\cong\mathbb{P}^{2}$, $Q^{(0)}$ is plane quintic curve, $L^{(0)}$ is a line.

Firstly, we show that:

\begin{thm}
    The K-moduli space $\mathbb{P}^{K}_{a,b}$ is non empty if and only if $5a\ge2b$.
\end{thm}

\begin{proof}
    By Theorem \ref{degeneration}, if $[(X;Q,L)]\in\overline{P}^K_{a,b}$, then $X$ can only possible be $\mathbb{P}^2$, $\mathbb{P}(1,1,4)$, $X_{26}$ and $\mathbb{P}(1,4,25)$. For the later three surfaces, we will show in Theorem \ref{whenp114appear}, Theorem \ref{whenX26appear}
    and Theorem \ref{whenP1425appear} that they are unstable if $5a<2b$. For the remaining case that $X\cong\mathbb{P}^2$, consider the divisorial valuation defined by $L$. By computation we have: $A(L)=1-a\ord_{L}(Q)-b\ord_{L}(L)$ and $S(L)=1-\dfrac{5}{3}a-\dfrac{1}{3}b$. If $(\mathbb{P}^2,aQ+bL)$ is K-semistable, then by valuative criterion:
    \begin{equation*}
        1-b\ge A(L)\ge S(L)=1-\dfrac{5}{3}a-\dfrac{1}{3}b
    \end{equation*}
    which equals to $5a\ge2b$.
 \end{proof}

 Then we will show that there indeed exists the $(a_0,b_0)$-K-semistable log Fano pair $(\mathbb{P}^2;Q,L)$ where $5a_0=2b_0$. In fact, we will find out all critical K-polystable pair associate to $\mathbb{P}^2$:

\begin{thm}
    The walls and the corresponding critical pairs associated to $\mathbb{P}^2$ are listed in the following Table\ref{KwallP2}:
    \begin{center}
    \renewcommand*{\arraystretch}{1.2}
    \begin{table}[ht]
    \centering
      \begin{tabular}{ |c |c |c|}
    \hline
     wall &  curve $Q^{(0)}$ on $\mathbb{P}^2$ & curve $L^{(0)}$ on $\mathbb{P}^2$  \\  \hline
     
      $2b=5a$   &  $\left\{zxy(x+\lambda_1y)(y+\lambda_2x)=0\right\}$   & $\left\{z=0\right\}$ \\ \hline 
      
     $5b=11a$  &  $\left\{x^2(y^3+x^2z)=0\right\}$   & $\left\{z=0\right\}$\\ \hline   
     $b=2a$ &     $\left\{yx^2(zx+y^2)=0\right\}$   & $\left\{z=0\right\}$ \\ \hline  
     $7b=13a$ &     $\left\{xy(y^3+zx^2)=0\right\}$   & $\left\{z=0\right\}$ \\ \hline
     $4b=7a$ &     $\left\{x^2(y^3+z^2x)=0\right\}$   & $\left\{z=0\right\}$ \\ \hline
     $3b=5a$ &     $\left\{y^5+x^4z=0\right\}$   & $\left\{z=0\right\}$ \\ \hline
     $5b=8a$ &     $\left\{xy(y^3+x^2z)=0\right\}$   & $\left\{z=0\right\}$ \\ \hline
     $7b=10a$ &     $\left\{y(y^4+x^3z)=0\right\}$   & $\left\{z=0\right\}$ \\ \hline
     
     $a=b$ &     $\left\{zxy(x+\lambda_1y)(x+\lambda_2y)=0\right\}$ or $\left\{y^2(y^3+x^2z)\right\}$  & $\left\{z=0\right\}$ \\ \hline
     
     $8b=5a$ &     $\left\{y^5+x^3z^2=0\right\}$   & $\left\{z=0\right\}$ \\ \hline
     $5b=2a$ &     $\left\{xz(y^3+x^2z)=0\right\}$   & $\left\{z=0\right\}$ \\ \hline
     $4b=a$ &     $\left\{yx(y^3+z^2x)=0\right\}$   & $\left\{z=0\right\}$ \\ \hline
     $7b=a$ &     $\left\{z(y^4+x^3z)=0\right\}$   & $\left\{z=0\right\}$ \\ \hline
    
     \end{tabular}
     \caption{ K-moduli walls from $\mathbb{P}^2$}
     \label{KwallP2}
\end{table}
\end{center}
\end{thm}

\begin{proof}
    Firstly, we assume $(\mathbb{P}^2,aQ^{(0)}+bL^{(0)})$ is a toric pair. The 2 dimensional torus $\mathbb{T}\cong\mathbb{C}^*\times\mathbb{C}^*$ act on $\mathbb{P}^2$ is defined by:
    \[
    (s,t)\circ[x:y:z]=[s^{\mu_1}t^{\lambda_1}x:s^{\mu_2}t^{\lambda_2}y:s^{\mu_3}t^{\lambda_3}z],\quad\forall(s,t)\in\mathbb{T}
    \]
    Note that $\mathbb{P}^2$ has no invariant curve with degree greater than one. In fact, we can assume $0=\lambda_3=\min\left\{\lambda_1,\lambda_2,\lambda_3\right\}$, $0=\mu_3=\min\left\{\mu_1,\mu_2,\mu_3\right\}$. If $x^iy^jz^k$ and $x^{i'}y^{j'}z^{k'}$ are two monimals in the equation of some irreducible invariant curve $C$, then we have:
    \[
    \left\{
    \begin{array}{cc}
       (i-i')\lambda_1+(j-j')\lambda_2=0 &  \\
       (i-i')\mu^1+(j-j')\mu_2=0  & 
    \end{array}
    \right.
    \]
    which implies that $i=i'$ and $j=j'$, so $C$ is a line defined by coordinates.

    Thus $Q^{(0)}$ is the union of $H_x$, $H_y$ and $H_z$, and $L^{(0)}=H_x$, $H_y$ or $H_z$. Moreover, both of $H_x$, $H_y$, $H_z$ are horizontal divisors.

    By computation, we have
   \begin{align}
       A(H_i)=&1-a\ord_{H_i}(Q^{(0)})-b\ord_{H_i}(L^{(0)}),\\
       S(H_i)=&1-\frac{5}{3}a-\frac{1}{3}b.
   \end{align}
   If $(\mathbb{P}^2,aQ^{(0)}+bL^{(0)})$ is K-polystable, we may assume $\ord_{H_z}(L^{(0)})=1$, then $\ord_{H_z}(L^{(0)})=\ord_{H_y}(L^{(0)})=0$ and $\ord_{H_z}(Q^{(0)})\le1$, $\ord_{H_x}(Q^{(0)})\ge2$, $\ord_{H_y}(Q^{(0)})\ge2$, which implies that $Q^{(0)}$ should be $\left\{zx^2y^2=0\right\}$. In this case the corresponding wall is $a=b$.

    From now on let's consider $(\mathbb{P}^2,aQ^{(0)}+bL^{(0)})$ as a strictly complexity one pair. Assume $\mathbb{C}^*$ acts on $\mathbb{P}^2$ defined by $t\circ[x:y:z]=[t^{m_1}x:t^{m_2}y:t^{m_3}z]$, $\forall t\in\mathbb{C}^*$.

    If $\exists i\neq j\in\left\{1,2,3
    \right\}$ such that $m_i=m_j$. Without generality, we may assume that $m_1=m_2>m_3=0$. Then the above $\mathbb{C}^*$-action can be written as $t\circ[x:y:z]=[x:y:t^{m}z]$, for some $m>0$. In this case, the only horizontal divisor is $H_{z}$. A direct computation shows that 
    \begin{align*}
         A(H_z)=&1-a\ord_{H_z}(Q^{(0)})-b\ord_{H_z}(L^{(0)}),\\
       S(H_z)=&1-\frac{5}{3}a-\frac{1}{3}b.
    \end{align*}
    If $(\mathbb{P}^2,aQ^{(0)}+bL^{(0)})$ is K-polystable, using theorem \ref{complexity1} we know that $\beta(H_z)=0$, which give us:
    \begin{equation}\label{c1}
        (\ord_{H_z}(Q^{(0)})-\frac{5}{3})a+(\ord_{H_z}(L^{(0)})-\frac{1}{3})b=0
    \end{equation}

    On the other hand, any line passing through the point $[0:0:1]$ is a vertical divisor. For any such line $H$, we can compute that:
    
    \begin{align*}
         A(H)=&1-a\ord_{H}(Q^{(0)})-b\ord_{H}(L^{(0)}),\\
       S(H_z)=&1-\frac{5}{3}a-\frac{1}{3}b.
    \end{align*}

    By theorem \ref{complexity1}, these divisors satisfied that $A(H)>S(H)$. Thus 
    \begin{equation}\label{c1'}
        (\frac{5}{3}-\ord_{H_z}(Q^{(0)}))a+(\frac{1}{3}-\ord_{H_z}(L^{(0)}))b>0
    \end{equation}

    Under the constraints (\ref{c1}) and (\ref{c1'}), we can find out all critical pairs and the corresponding walls with torus action of above type: 

    \begin{enumerate}
        \item If $\ord_{H_{z}}(L^{(0)})=1$. Then $\ord_{H}(L^{(0)})=0$ for $H\neq H_z$. From (\ref{c1}) we deduce that $\ord_{H_{z}}(Q^{(0)})\le1$. 
        \begin{enumerate}
            \item If $\ord_{H_{z}}(Q^{(0)})=1$, we can see that (\ref{c1}) implies that $a=b$, and conditions (\ref{c1'}) implies that $\ord_{H}(Q^{(0)})\le1$. Thus we conclude that $(\mathbb{P}^2,a(Q^{(0)}_0\cup H_z)+b\left\{z=0\right\})$ is a critical pair with the critical slope $a=b$, where $Q^{(0)}_0$ is the union of 4 distinct lines and both of them passing through the point $[0:0:1]$.
            \item If $\ord_{H_{z}}(Q^{(0)})=0$, we can see that (\ref{c1}) implies that $2b=5a$, and conditions (\ref{c1'}) implies that $\ord_{H}(Q^{(0)})\le2$. Thus we conclude that $(\mathbb{P}^2,a(Q^{(0)}_0\cup H_z)+b\left\{z=0\right\})$ is a critical pair with the critical slope $a=b$, where $Q^{(0)}_0$ is the union of 4 lines and both of them passing through the point $[0:0:1]$. Moreover, at most two couple of them identical.
        \end{enumerate}
        \item If $\ord_{H_{z}}(L^{(0)})=0$. Then $\ord_{H_0}(L^{(0)})=1$ for some $H_0\neq H_z$. For $H_{0}$, we have:
        \begin{equation}\label{c''}
            (\frac{5}{3}-\ord_{H_0})a-\frac{2}{3}b>0
        \end{equation}
        which implies that $\ord_{H_0}(Q^{(0)})\le1$.
        Moreover, (\ref{c1}) implies that $\ord_{H_z}(Q^{(0)})\ge2$. 
        \begin{enumerate}
            \item If $\ord_{H_z}(Q^{(0)})=2$. By (\ref{c''}), we can deduce that $\ord_{H_0}(Q^{(0)})=0$, and by (\ref{c1'}) we know that $\ord_{H}(Q^{(0)})\le1$. Such pair has already appear in the above discussion, which corresponding the critical slope is $a=b$.
            \item If $\ord_{H_z}(Q^{(0)})\ge3$, it is easy to check that the (\ref{c''}) is not compatible with (\ref{c1}) and (\ref{c1'}). Thus there is no critical pair satisfied these conditions.
        \end{enumerate}
    \end{enumerate}

    If $m_i\neq m_j$ for $\forall i\neq j\in\left\{1,2,3\right\}$. We can assume that $m_1>m_2>m_3$. The torus action $\lambda$ can be rewritten as $t\circ[x:y:z]=[t^{n_1}x:t^{n_2}y:z]$ with $n_1$, $n_2>0$. In this case, there is no horizontal divisor under this torus action. To figure out the remaining (potential) walls, we should compute the Futaki character $Fut(\lambda)$ of this torus action. Let
    \[
    \pi:Y\rightarrow X=\mathbb{P}^2
    \]
    be the weighted blow up at point $[0:0:1]$ along the local coordinate $(\frac{x}{z},\frac{y}{z})$ with weight $(n_1,n_2)$. Denote the exceptional divisor by $E$. By \cite{Fujita}, we know that $Fut(\lambda)=\beta(E)$, where $E$ is the exceptional divisor over $\mathbb{P}^2$ induced by $\lambda$. A direct computation shows that
    \begin{align*}
         A(E)=&n_1+n_2-a\ord_{E}(Q^{(0)})-b\ord_{E}(L^{(0)}),\\
         S(E)=&(n_1+n_2)(1-\frac{5}{3}a-\frac{1}{3})b.
    \end{align*}
    If $(\mathbb{P}^2,aQ^{(0)}+bL^{(0)})$ is K-polystable, then $Fut(\lambda)=0$ which implies that
    \begin{equation}\label{veryimportant}
        (\frac{5}{3}(n_1+n_2)-\ord_{E}(Q^{(0)}))a+(\frac{1}{3}(n_1+n_2)-\ord_{E}(L^{(0)}))b=0
    \end{equation}
    We give a algorithm about finding critical pairs and the corresponding walls as follows:

    Step 0. We can observe that $\ord_{E}(Q^{(0)})>0$, so the local equation of $Q^{(0)}$ in $\left\{z=1\right\}$ has no constant terms. 

    Step 1. By above observation, we can choose any two monimals says $x^iy^j$ and $x^{i'}y^{j'}$ appeared in the local equations of $Q^{(0)}$ in the open subset $\left\{z=1\right\}$. By $\mathbb{C}^*$-invariance, they should satisfied equation $in_1+jn_2=i'n_1+j'n_2$, where $i$, $j$, $i'$ and $j'\in\left\{0,1,...,5\right\}$;

    Step 2. Note that $\ord_{E}(Q^{(0)})=in_1+jn_2$, and $\ord_{E}(L^{(0)})\in\left\{0,n_1,n_2,2n_1,2n_2,\right\}$. Let $(i,j)$, $(i',j')$ and $\ord_{E}(L^{(0)})$ runs over all possibility, put them into (\ref{veryimportant}), we will obtain all the (potential) critical lines and the corresponding (potential) critical pairs. 
    
    Step 3. Note that $\ord_{H_{i}}(Q^{(0)})$ and $\ord_{H_{i}}(L^{(0)})$ are also fixed according to the chosen the monimals $x^iy^j$, $x^{i'}y^{j'}$ and $\ord_{E}(L)$ in Step 1 and Step 2 where $H_{i}\in\left\{H_x,H_y,H_z\right\}$. Since they are vertical divisors, so
    \begin{equation}\label{check}
        1-a\ord_{H_i}(Q^{(0)})-b\ord_{H_i}(L^{(0)})>1-\frac{5}{3}a-\frac{1}{3}b
    \end{equation}
    Using (\ref{check}), we can throw away the fake lines, i.e. the lines we got from the Step 2 which is not compatible with condition (\ref{check}).

    Step 4. The remaining vertical divisors is depending on the equation of $Q^{(0)}$ we have in Step 3. But this can be easily check case by case. 

    Using above algorithm, we finally find out all critical pairs and the critical walls as follows:
    \begin{itemize}
        \item $Q^{(0)}$ is irreducible degree 5 curve:
    
    \begin{align*}
        \text{$Q^{(0)}=\left\{y^5+x^4z=0\right\}$ or $\left\{y^5+xz^4=0\right\}$, $L^{(0)}=H_z$};\quad&3b=5a\\
        \text{$Q^{(0)}=\left\{y^5+x^3z^2=0\right\}$ or $\left\{y^5+x^2z^3=0\right\}$,  $L^{(0)}=H_z$};\quad&8b=5a  
    \end{align*}

        \item $Q^{(0)}$ is the union of irreducible degree 4 curve with a line:

        \begin{align*}
        \text{$Q^{(0)}=\left\{z(y^4+x^3z)=0\right\}$ or $\left\{x(y^4+xz^3)=0\right\}$, $L^{(0)}=H_z$ or $H_x$};\quad&7b=a\\
        \text{$Q^{(0)}=\left\{y(y^4+x^3z)=0\right\}$ or $\left\{y(y^4+xz^3)=0\right\}$,  $L^{(0)}=H_z$ or $H_x$};\quad&7b=10a\\  
        \text{$Q^{(0)}=\left\{x(y^4+x^3z)=0\right\}$ or $\left\{z(y^4+xz^3)=0\right\}$,  $L^{(0)}=H_z$ or $H_x$};\quad&7b=13a
    \end{align*}
        \item  $Q^{(0)}$ is the union of irreducible degree 3 curve with two lines:
    \begin{align*}
        \text{$Q^{(0)}=\left\{z^2(y^3+x^2z)=0\right\}$ or $\left\{x^2(y^3+xz^3)=0\right\}$, $L^{(0)}=H_x$ or $H_z$};\quad&4b=7a\\
        \text{$Q^{(0)}=\left\{y^2(y^3+x^2z)=0\right\}$ or $\left\{y^2(y^3+xz^2)=0\right\}$,  $L^{(0)}=H_z$ or $H_x$};\quad&b=a\\  
        \text{$Q^{(0)}=\left\{x^2(y^3+x^2z)=0\right\}$ or $\left\{z^2(y^3+xz^2)=0\right\}$,  $L^{(0)}=H_z$ or $H_x$};\quad&5b=11a\\
        \text{$Q^{(0)}=\left\{xy(y^3+x^2z)=0\right\}$ or $\left\{yz(y^3+xz^2)=0\right\}$,  $L^{(0)}=H_z$ or $H_x$};\quad&5b=8a\\
        \text{$Q^{(0)}=\left\{xz(y^3+x^2z)=0\right\}$ or $\left\{xz(y^3+xz^2)=0\right\}$,  $L^{(0)}=H_z$ or $H_x$};\quad&5b=2a\\
        \text{$Q^{(0)}=\left\{yz(y^3+x^2z)=0\right\}$ or $\left\{xy(y^3+xz^2)=0\right\}$,  $L^{(0)}=H_x$ or $H_z$};\quad&4b=a\\
    \end{align*}  

        \item  $Q^{(0)}$ is the union of irreducible degree 2 curve with three lines:

        \begin{equation*}
            Q^{(0)}=\left\{yx^2q=0\right\},\quad \text{$L^{(0)}=H_z$ or $H_y$};\quad b=2a
        \end{equation*}
    \end{itemize}

    Finally, by \cite{kollár1996singularitiespairs}(Proposition 8.14), we can taking the weighted blowup as above and check case by case to show that 
    \begin{equation}
        \lct(\mathbb{P}^2,Q^{(0)}+tL^{(0)})=\dfrac{3}{5+t},
    \end{equation}
    where $t=\dfrac{b}{a}$ is the slope of above critical lines. Thus the end point of above lines are all in the boundary of log Fano domain.
    
    The proof is finished.
\end{proof}



\subsection{The critical $\mathbb{P}(1,1,4)$-pair}

In this case, $X\cong\mathbb{P}(1,1,4)$, $Q^{(1)}\sim\mathcal{O}(10)$, $L^{(1)}\sim\mathcal{O}(2)$, $D=a\cdot Q^{(1)}+b\cdot L^{(1)}$, where $5a+b\le3$. Let $x,y,z$ be homogeneous coordinate on $\mathbb{P}(1,1,4)$ with weight $1,1,4$ respectively.

\begin{thm}\label{whenp114appear}
    Assume the equation of $Q^{(1)}$ is $\left\{z^2f_2(x,y)+zf_6(x,y)+f_{10}(x,y)=0\right\}$, where $f_{i}(x,y)$ is homogeneous polynomial of $x,y$ of degree $i$. Then we have 
    \begin{enumerate}
        \item If $(X,D)$ is K-semistable then $7a-b\ge3$. 
        \item If $f_2$ is rank 1 and $(X,D)$ is K-semistable, then $11a+b\ge6$.
        \item if $f_2=0$, then $(X,D)$ is K-unstable for any $(a,b)$.
    \end{enumerate}
\end{thm}

\begin{proof}
    Let's take weighted blow up at $[0:0:1]$ with weight $(1,1,0)$ along the coordinate $[x:y:z]$:
    \[
    \pi:Y\rightarrow\mathbb{P}(1,1,4).
    \]
    Denote its exceptional divisor by $E$. Then we can compute that:
    \begin{align}
        A_{(X,D)}(E)&=\dfrac{1}{2}-a\cdot\coeff_{E}(\pi^*Q)-b\cdot\coeff_{E}(\pi^*L),\\
        S_{(X,D)}(E)&=1-\frac{5}{3}a-\frac{1}{3}b.
    \end{align}
    By the equation of $Q^{(1)}$ and $L^{(1)}$, we note that they are both passing through the point $[0:0:1]$, thus $\coeff_{E}(\pi^*Q)\ge\frac{1}{2}$ and $\coeff_{E}(\pi^*L)\ge\frac{1}{2}$. If $(X,D)$ is K-semistable, then we have
    \[
    1-\frac{5}{3}a-\frac{1}{3}b=S_{(X,D)}(E)\le A_{(X,D)}(E)\le\dfrac{1}{2}(1-a-b)
    \]
    which implies that $7a-b\ge3$. 

    Moreover, if the equation of $Q^{(1)}$ is $\left\{zf_{6}(x,y)+f_{10}(x,y)=0\right\}$, then $\coeff_{E}(\pi^*Q^{(1)})\ge6$, by the same arguments, we deduce that $a-b\ge3$, which has empty intersection with the origin $\left\{5a+b\le3,a\ge0,b\ge0\right\}$. Thus in this case $(X,D)$ is K-unstable. Finally, let's consider the case that the equation of $Q^{(1)}$ is 
    \[
    \left\{z^2x^2+zf_6(x,y)+f_{10}(x,y)=0\right\}.
    \]
    Let $\pi:\mathbb{A}^{2}_{(x,y)}\rightarrow X$ be the cyclic quotient map defined by $\pi(x,y)=[x,y,1]$. Let $v:\pi_{*}\Tilde{v}$ where $\Tilde{v}$ is the monomial valuation on $\mathbb{A}^{2}_{(x,y)}$ of weight $(3,1)$. Note that $v(Q^{(1)})\ge6$ and $v(L^{(1)})\ge2$. By computation we show that $S(v)=8(1-\frac{5}{3}a-\frac{1}{3}b)$. If $(X,D)$ is K-semistable, by valuative criterion we have:
    \begin{equation}
        4-6a-2b\ge8(1-\frac{5}{3}a-\frac{1}{3}b)
    \end{equation}
    which implies $11a+b\ge6$.
    
\end{proof}

\begin{thm}\label{4.8}
    The walls and the corresponding critical pairs associated to $\mathbb{P}(1,1,4)$ are listed in the following table:
    \begin{center}
    \renewcommand*{\arraystretch}{1.2}
    \begin{table}[ht]
    \centering
      \begin{tabular}{ |c |c |c|}
    \hline
     wall &  curve $Q^{(1)}$ on $\mathbb{P}(1,1,4)$ & curve $L^{(1)}$ on $\mathbb{P}(1,1,4)$  \\  \hline 

     $7a-b=3$   &  $\left\{z^2xy=0\right\}$   & $\left\{xy=0\right\}$ or $\left\{l_1l_2=0\right\}$\\ \hline 
     $35a-17b=15$  &  $\left\{z^2xy+y^{10}=0\right\}$   & $\left\{x^2=0\right\}$\\ \hline 
     $7a-4b=3$ &     $\left\{z^2xy+\lambda zx^5y=0\right\}$   & $\left\{y^2=0\right\}$ \\ \hline 
     $11a+b=6$ &     $\left\{x^2z^2+y^6z=0\right\}$   & $\left\{y^2=0\right\}$ \\ \hline
     $11a-2b=6$ &     $\left\{x^2z^2+y^6z=0\right\}$   & $\left\{xy=0\right\}$ \\ \hline
     $11a-5b=6$ &     $\left\{x^2z^2+y^6z=0\right\}$   & $\left\{x^2=0\right\}$ \\ \hline
     \end{tabular}
     \caption{ K-moduli walls from $\mathbb{P}(1,1,4)$}
     \label{Kwall3}
\end{table}
\end{center}
where $l_1$ and $l_2$ are both degree 1 form of $x$, $y$, such that $l_i\neq x$ or $y$.
\end{thm}

\begin{proof}
    \emph{Case 1: Toric Case.} 
    
    Firstly, we assume $(\mathbb{P}(1,1,4),aQ^{(1)}+bL^{(1)})$ is a toric pair. After a suitable coordinate change, the 2 dimensional torus $\mathbb{T}\cong\mathbb{C}^*\times\mathbb{C}^*$ act on $\mathbb{P}(1,1,4)$ is can be written by:
    \[
    (s,t)\circ[x:y:z]=[s^{\mu_1}t^{\lambda_1}x:s^{\mu_2}t^{\lambda_2}y:s^{\mu_3}t^{\lambda_3}z],\quad\forall(s,t)\in\mathbb{T}.
    \]
    It is easy to see that $H_{x}$, $H_{y}$, $H_z$ are only $\mathbb{T}$-invariant divisor on $\mathbb{P}(1,1,4)$. In particular, they are both horizontal divisors. We can compute that 
    \[
    S(H_{x})=S(H_{y})=2(1-\dfrac{5}{3}a-\dfrac{1}{3}b),\quad S(H_z)=\frac{1}{2}(1-\dfrac{5}{3}a-\dfrac{1}{3}b).
    \]
    If $(\mathbb{P}(1,1,4),aQ^{(1)}+bL^{(1)})$ is K -polystable, by theorem \ref{complexity1} we have
    \begin{align}
        A(H_{x})=1-a\ord_{H_x}(Q^{(1)})-b\ord_{H_x}(L^{(1)})=2(1-\frac{5}{3}a-\frac{1}{3}b)\label{w2'}\\
        A(H_{z})=1-a\ord_{H_z}(Q^{(1)})-b\ord_{H_z}(L^{(1)})=\frac{1}{2}(1-\frac{5}{3}a-\frac{1}{3}b)\label{w2}
    \end{align}

    By (\ref{w2}), we get $(\ord_{H_z}(Q^{(1)})-\frac{5}{6})a+(\ord_{H_{z}}(L^{(1)})-\frac{1}{6})b=\frac{1}{2}$. Since $\ord_{H_z}(L^{(1)})=0$, we know that $\ord_{H_z}(Q^{(1)})\ge1$. If $\ord_{H_z}(Q^{(1)})=1$, this implies $a-b=3$ which is contradict to $5a+b\le3$.
    If $\ord_{H_z}(Q^{(1)})=2$, we have $7a-b=3$. By (\ref{w2'}), the corresponding critical pair is $(\mathbb{P}(1,1,4),a\cdot\left\{xyz^2=0\right\}+b\cdot\left\{xy=0\right\})$. 

    \emph{Case 2: Complexity One Case.} 
    
    From now on, let's consider $(\mathbb{P}(1,1,4),aQ^{(1)}+bL^{(1)})$ as strictly complexity one pairs, that is, the dimension of the maximal torus is $1$. Assume the $\mathbb{C}^*$ action on $\mathbb{P}(1,1,4)$ is defined by $t\circ[x:y:z]=[t^{m_1}x:t^{m_2}y:t^{m_3}z]$, $\forall t\in\mathbb{C}^*$.

    \emph{Case 2.1} If $4m_2\neq 4m_1=m_3$, we can rewrite the torus action by 
    \[
    t\circ[x:y:z]=[x:t^my:z]\]
    where $m=m_2-m_1$. In this case, $H_y=\left\{y=0\right\}$ is a horizontal divisor. The vertical divisors are $H_{x}$, $H_{z}$ and $C_{\lambda}:=\left\{z+\lambda\cdot x^4=0|\lambda\in\mathbb{C}\right\}$.
    If $(\mathbb{P}(1,1,4),aQ^{(1)}+bL^{(1)})$ is K-polystable, by theorem \ref{complexity1}, we get:
    \begin{equation}\label{4.13}
    \left\{
    \begin{array}{l}
        (\frac{10}{3}-\ord_{H_y}(Q^{(1)}))a+(\frac{2}{3}-\ord_{H_y}(L^{(1)}))b=1\\
        (\frac{10}{3}-\ord_{H_x}(Q^{(1)}))a+(\frac{2}{3}-\ord_{H_x}(L^{(1)}))b>1\\
        \frac{1}{2}>(\ord_{C}(Q^{(1)})-\frac{5}{6})a-\frac{1}{6}b
    \end{array} 
    \right.
    \end{equation}
    \begin{itemize}
        \item If $L^{(1)}=H_{x}+H_{y}$, then $\ord_{H_y}(L^{(1)})=\ord_{H_x}(L^{(1)})=1$. By above constraints, we deduce that $\ord_{H_y}(Q^{(1)})\le2$ and $\ord_{H_x}(Q^{(1)})\le2$.
        \begin{itemize}
            \item  If $\ord_{H_x}(Q^{(1)})=2$, by (\ref{4.13}) we deduce that $4a-b>3$ which is contradict to the original constraints $5a+b\le3$ and $a\ge0$, $b\ge0$.
            \item If $\ord_{H_x}(Q^{(1)})=0$, by (\ref{4.13}) we deduce that $10a-b=3$ and $\ord_{H_y}(Q^{(1)})=1$. But this is a contradiction since we have shown in theorem\ref{whenp114appear} that the $\mathbb{P}(1,1,4)$-pair appears in the K-moduli space only if $7a-b\ge3$.
 
            \item If $\ord_{H_x}(Q^{(1)})=1$, by similarly method we deduce that $\ord_{H_y}(Q^{(1)})=1$. But these two constraints are contradict to each other. So this case is excluded.       \end{itemize}
        \item If $L^{(1)}=2H_{x}$, then $\ord_{H_y}(L^{(1)})=0$, $\ord_{H_x}(L^{(1)})=2$. The above constraints can be rewritten as
    
    \begin{equation}\label{1}
        (\frac{10}{3}-\ord_{H_y}(Q^{(1)}))a+\frac{2}{3}b=1
    \end{equation} 
    
    \begin{equation}\label{2}
        (\frac{10}{3}-\ord_{H_x}(Q^{(1)}))a-\frac{4}{3}b>1
    \end{equation}
        \begin{itemize}
            \item If $\ord_{H_y}(Q^{(1)})=4$, by (\ref{1}), we deduce that $2b-2a=3$. But we need $7a-b\ge3$, so this case is excluded. 
            \item If $\ord_{H_y}(Q^{(1)})\le3$, the same arguments as above also deduce contradictions.
        \end{itemize}

        \item If $L^{(1)}=2H_{y}$. By the same arguments as above, we conclude that the only possible case is $Q^{(1)}=\left\{xyz(z+x^4)=0\right\}$,  $L^{(1)}=\left\{y^2=0\right\}$ and the corresponding critical line is $7a-4b=3$.
    \end{itemize}

    \emph{Case 2.2} If $m_1=m_2\neq 4m_3$. The torus action can be written as 
    \[
    t\circ[x:y:z]=[x:y:t^mz].
    \] 
    Note that in this case, any degree 1 curves on $\mathbb{P}(1,1,4)$ are vertical, and there is only one horizontal divisor $H_z=0$. So $(\mathbb{P}(1,1,4),aQ^{(1)}+bL^{(1)})$ is K-polystable implies that 
    \[
    \frac{1}{2}=(\ord_{H_{z}}(Q^{(1)})a-\frac{5}{6})-\frac{1}{6}b.
    \]
    Thus $\ord_{H_{z}}(Q^{(1)})=1$ or $2$. But $\ord_{H_{z}}(Q^{(1)})\neq1$, otherwise we will deduce that $b=a-3$, which is a contradiction. Thus $\ord_{H_z}(Q^{(1)})=2$. The equation of $Q^{(1)}$ can be assume to be $z^2q(x,y)=0$, where $q(x,y)$ is a quartic form of $x$, $y$. 
    
    If the rank of $q$ is 1, we can rewrite the equation of $Q$ by $z^2x^2=0$ after a suitable coordinate change. But the following constraint
    \[
    (\frac{10}{3}-\ord_{H_x}(Q^{(1)}))a+(\frac{2}{3}-\ord_{H_x}(L^{(1)}))b>1
    \]
    implies that $4a+2b>3$ which is contradicts to the condition $5a+b\le3$ and $7a-b=3$. 
    
    If the rank of $q$ is 2, we can rewrite the equation of $Q^{(1)}$ by $z^2xy=0$ after a suitable coordinate change. Assume the equation of $L^{(1)}$ is $\left\{l_1l_2=0\right\}$. If $Q^{(1)}$ and $L^{(1)}$ have common components, for example, we may assume that $l_1=x$. Then $l_2$ must equals to $y$. Otherwise, $H_x$ is a vertical divisor, and by the following constraints:
    \[
    (\frac{10}{3}-\ord_{H_x}(Q^{(1)}))a+(\frac{2}{3}-\ord_{H_x}(L^{(1)}))b>1
    \]
    we deduce that $7a-b>3$, which contradicts to the condition $7a-b=3$.
    
    Next, let's consider the case that $Q^{(1)}$ and $L^{(1)}$ has no common components. It is easy to check that all constraints from the vertical divisors and horizontal are compatible, thus on the critical line $7a-b=3$, we find the another kind of critical pairs: $(\mathbb{P}(1,1,4);Q^{(1)}=\left\{z^2xy=0\right\},L^{(1)}=\left\{l_1l_2=0\right\})$ where $l_1$ and $l_2$ are linear forms of $x,y$ and $l_i\neq x$ or $y$.

    \emph{Case 2.3} Finally, let's consider the remaining case that $m_1\neq m_2\neq4m_3$. Under suitable coordinate change, we can always assume that the torus action is
    \[
    t\circ[x:y:z]=[x:t^{n_1}y:t^{n_2}z],\quad \forall t\in\mathbb{C}^*,
    \]
    in which $n_1,n_2>0$, and $n_2\neq4n_1$.
    In this case, $H_x$, $H_y$ and $H_z$ are both vertical divisors. Moreover, there could be some irreducible invariant curve of degree greater than 1. We denote such curves by $C$.

    Let
    \[
    \pi:Y\rightarrow X=\mathbb{P}(1,1,4)
    \]
    be the weighted blow up at point $[1:0:0]$ with respect to the coordinate $(\frac{y}{x},\frac{z}{x^4})$ with weight $(n_1,n_2)$.
    Denote the exceptional divisor by $E$. Then we have:
    \begin{align}
        A(E)=&n_1+n_2-a\ord_{E}(Q^{(1)})-b\ord_{E}(L^{(1)})\\
        S(E)=&(\frac{1}{2}n_2+2n_1)(1-\frac{5}{3}a-\frac{1}{3}b).
    \end{align}
    and $Fut(\lambda)=\beta(E)=A(E)-S(E)$.

    If $(\mathbb{P}(1,1,4),aQ^{(1)}+bL^{(1)})$ is K- polystable, by theorem\ref{complexity1}, we deduce that
    \begin{equation}\label{important2}
    n_1-\frac{1}{2}n_2=(\frac{5}{6}n_2+\frac{10}{3}n_1-\ord_{E}(Q^{(1)})a+(\frac{1}{6}n_2+\frac{2}{3}n_1-\ord_{E}(L^{(1)}))b
    \end{equation}
    Let $Q$ be the irreducible component with highest degree of $Q^{(1)}$. 
    
    \begin{itemize}
        \item If $\deg Q=10$, then $Q^{(1)}=Q=\left\{zx^6+y^{10}=0\right\}$, and $\ord_{E}(Q)=10n_1$. By (\ref{important2}), we know that
        \[
        -4n_1=\frac{5}{3}n_1a+(\frac{7}{3}n_1-\ord_{E}(L^{(1)}))b.
        \]
        But $\ord_{E}(L^{(1)})\le2n_1$, so the above inequality is impossible!

        \item If $\deg(Q)=8$, $7$ or $5$, it is easy to exclude these cases by using the same arguments as above.

        \item If $\deg(Q)=9$, then the equation of $Q$ can be written as $\left\{z^2x^2+y^9=0\right\}$. 
        
        If $Q=\left\{z^2x^2+y^9=0\right\}$, then $2n_2=9n_1$ and $\ord_{E}(Q^{(1)})=10n_1$.
        \begin{itemize}
            \item If $Q^{(1)}=Q\cup H_x$, then $\ord_{E}(Q^{(1)})=2n_2$.
            \begin{itemize}
                \item If $L=2H_x$, then (\ref{important2}) can be rewritten as $23a=17b+3$.
                But we also have 
                \[
                (\frac{10}{3}-\ord_{H_x}(Q^{(1)}))a+(\frac{2}{3}-\ord_{H_x}(L^{(1)}))b>1\Rightarrow\quad 7a-4b>3
                \]
                which is impossible.
                \item If $L^{(1)}=H_x+H_y$ or $L^{(1)}=2H_{y}$, by the same arguments as above we can conclude that these cases are both impossible.
            \end{itemize}
            \item If $Q^{(1)}=Q\cup H_y$, then $\ord_{E}(Q^{(1)})=10n_1$.
            \begin{itemize}
                \item If $L^{(1)}=2H_x$, then $\ord_{E}(L^{2})=0$, and the (\ref{important2}) equals to 
                \[
                35a-17b=15
                \]
                On the other hand, for vertical divisors:
                \begin{align}
                    (\frac{10}{3}-\ord_{H_x}(Q^{(1)}))a+(\frac{2}{3}-\ord_{H_x}(L^{(1)}))b>1\Leftrightarrow&\quad 10a-4b>3\\
                    (\frac{10}{3}-\ord_{H_y}(Q^{(1)}))a+(\frac{2}{3}-\ord_{H_y}(L^{(1)}))b>1\Leftrightarrow&\quad 7a+2b>3\\
                    \frac{1}{2}>(\ord_{H_z}(Q^{(1)}-\frac{5}{6}))a-\frac{1}{6}b\Leftrightarrow&\quad \frac{1}{2}>-\frac{5}{6}a-\frac{1}{6}b
                \end{align}
                All the above conditions are compatible, thus we conclude that $((\mathbb{P}(1,1,4),a\left\{(xz^2+y^9)y=0\right\}+b\left\{x^2=0\right\}$ is a critical pair and the corresponding critical line is $35a-17b=15$.
                \item If $L=H_x+H_y$ or $L=2H_y$, we can exclude these cases by using the same arguments as above.
            \end{itemize}
        \end{itemize}
        
        \item If $\deg(Q)=6$, we can assume that $Q^{(1)}=Q\cup C$ where $\deg(C)=4$. It is easy to check that $C$ can not be the union of lines, thus $C$ is an irreducible degree 4 curve. Since it is  torus invariant divisor, we conclude that $C=\left\{z=0\right\}$, thus $Q^{(1)}=\left\{z(zx^2+y^6)=0\right\}$. Then $\ord_{E}(Q)=2n_2=12n_1$, $\ord_{H_x}(Q^{(1)})=\ord_{H_y}(Q^{(1)})=0$. Let $m=\dfrac{\ord_{E}(L^{(1)})}{n_1}$. It is easy to see that $m\in\left\{1,2,3\right\}$. By (\ref{important2}), we conclude that there are critical pairs as follows:
        \begin{align*}
           m=0&\quad\text{Which implies that $L^{(1)}=2H_x$ and $11a-5b=6$} \\
           m=1&\quad\text{Which implies that $L^{(1)}=H_x+H_y$ and $11a-2b=6$}\\
           m=2&\quad\text{Which implies that $L^{(1)}=2H_y$ and $11a+b=6$}
        \end{align*}
        For above pairs, the conditions on vertical divisors are easy to check.
        The proof is finished.
        
    \end{itemize}
    
\end{proof}

\subsection{The critical $X_{26}$-pair}

 Recall that $X=X_{26}$ is the surface which can be embedded in weighted projective space $\mathbb{P}(1,2,13,25)$. Let $k[x,y,z,w]$ be the homogeneous coordinate ring of $\mathbb{P}(1,2,13,25)$, then $X_{26}$ is defined by equation 
 \[
 \left\{xw=z^2+ty^{13}\right\}
 \]
 It can be seen as the partial smoothing of the $\frac{1}{4}(1,1)$-singularity on $\mathbb{P}(1,4,25)$. 
 Thus the only singularity of $X_{26}$ is of type $\frac{1}{25}(1,4)$.
 It easy to see that $K_{X_{26}}\sim\mathcal{O}_{\mathbb{P}(1,2,13,25)}(-15)|_{X}$. In this case $Q^{(2)}\sim\mathcal{O}_{25}$ is defined by $\left\{\lambda\cdot w^2+f(x,y)z+g(x,y)=0\right\}$ and $L^{(2)}\sim\mathcal{O}_{X}(5)$ is defined by $\left\{\lambda_1x^5+\lambda_2x^3y+\lambda_3xy^2=0\right\}$.

 Firstly, we need to analyze the feature of the curves $C^{(4)}$ on $K$-semistable log Fano pairs $(X_{26},aQ^{(2)}+bL^{(2)})$:

 \begin{thm}\label{whenX26appear}
   If $(X_{26},aQ^{(2)}+bL^{(2)})$ is K-semistable, then $Q^{(2)}$ is not passing through the singular point of $X_{26}$. Furthermore, the coefficients $(a,b)$ should satisfies the condition $b+15a\ge8$.
 \end{thm}

\begin{proof}
    We consider question in the affine open subset $\left\{w=1\right\}$. Then we have the cyclic quotient map $\pi:\mathbb{A}^{2}_{(y,z)}\rightarrow X_{26}$ definde by $\pi(y,z)=[y^{13}+z^{2},y,z,1]$. Let $F$ be the exceptional divisor on $\mathbb{A}^{2}_{(y,z)}$ given by the $(2,13)$-weighted blow up. Let $E$ be quotient of $F$ over $X_{26}$. Then it is clear that $\ord_{E}=\frac{1}{25}\pi^*\ord_{F}$. Easy computation shows that
    \[
    S(\ord_E)=\frac{9}{25}(15-25a-5b),
    \]
    and the log discrepancy is 
    \[
    A(\ord_E)=\frac{3}{5}-a\cdot\ord_{E}(Q^{(2)})-b\cdot\ord_{E}(L^{(2)}).
    \]
     If $Q^{(2)}$ passing through the singular point of $X_{26}$, then its equation can be written as $\left\{f(x,y)z+g(x,y)=0\right\}$ which implies that $\ord_{E}(Q^{(2)})\ge1$ and $\ord_{E}(L^{(2)})\ge\ord_{E}(5\left\{x=0\right\})=\dfrac{26}{5}$.
     
     On the other hand, if $(X_{26},aQ^{(2)}+bL^{(2)})$ is K-semistable. Then by Fujita-Li's criterion this implies 
     \[
     A(\ord_E)\ge S(\ord_E)
     \]
     which is equivalent to the following condition:
     \begin{equation}\label{X26Kss}
     \frac{5(9-\ord_{E}(C))}{24}a+\frac{9-5\ord_{E}(L^{(2)})}{24}b\ge1
     \end{equation}
     But this is contradict to the assumption $\left\{5a+3b\le1\right\}$!
     So we prove the first statements of the theorem.

     For the second statement, we only need to notice that the equation of $L^{(2)}$ must contains $x$, thus $A(\ord_{E})\le\dfrac{3}{5}-\dfrac{6}{5}b$. Combine with the condition(\ref{X26Kss}), we can deduce that $b+15a\ge8$.
\end{proof}

Let's denote $Q^{(2)}_{0}=\left\{w=0\right\}$. We will show that the position of $Q^{(2)}$ on $K$-polystable pair $(X_{26},aQ^{(2)}+bL^{(2)})$ is fixed:

\begin{thm}
    If $(X_{26},aQ^{(2)}+bL^{(2)})$ is K-polystable, then $Q^{(2)}=Q^{(2)}_{0}$.
\end{thm}

\begin{proof}
    Note that $(X_{26},aQ^{(2)}+bL^{(2)})$ admit a special degeneration to $(X_{26},aQ^{(2)}_{0}+bL^{(2)}_{0})$ via 1-PS defined by $t\circ[x,y,z,w]=[t^{26}x,t^2y,t^{13}z,w]$. Since $Fut=0$, we deduce that $(X_{26},aQ^{(2)}_{0}+bL^{(2)}_{0})$ is K-semistable. But $(X_{26},aQ^{(2)}+bL^{(2)})$ is K-polystable implies $(X_{26},aQ^{(2)}+bL^{(2)})\cong(X_{26},aQ^{(2)}_{0}+bL^{(2)}_{0})$.
\end{proof}

Finally, let's compute the K-stability threshold of $(X_{26},aQ^{(2)}_{0}+bL^{(2)}_{0})$. Since we have known that there exists a torus action on it, we can use the Theorem \ref{complexity1} to check the polystability of this pair. 

\begin{thm}
    The walls and the corresponding critical pairs associated to $X_{26}$ are listed in the following table:
    \begin{center}
    \renewcommand*{\arraystretch}{1.2}
    \begin{table}[ht]
    \centering
      \begin{tabular}{ |c |c |c|}
    \hline
     wall &  curve $Q^{(2)}_0$ on $X_{26}$ & curve $L^{(2)}_0$ on $X_{26}$  \\  \hline
      $45a-17b=24$   &  $\left\{w=0\right\}$   & $\left\{x^5=0\right\}$ 
      \\ \hline  
     $45a-7b=24$  &  $\left\{w=0\right\}$   & $\left\{x^3y=0\right\}$\\ \hline  
     $15a+b=8$ &     $\left\{w=0\right\}$   & $\left\{xy^2=0\right\}$ \\ \hline  

     \end{tabular}
     \caption{ K-moduli walls from  index $5$ del Pezzo  $X_{26}$}
     \label{Kwall4}
\end{table}
\end{center}
\end{thm}

\begin{proof}
    Note that there is no horizontal divisor under this action, so we only need to consider the Futaki character. By \cite{Fujita}, $Fut(\lambda)=\beta(E)=A(E)-S(E)$, here $S(E)=\frac{27}{5}-9a-\frac{9}{5}b$. The only things that affect the wall is the positions of line $L^{(2)}$. We compute it case by case:
    \begin{itemize}
        \item $L^{(2)}_{0}=\left\{x^5=0\right\}$. Then $A(E)=\frac{3}{5}-\frac{26}{5}b$. So $Fut=0\Rightarrow 45a-17b=24$,
        \item $L^{(2)}_{0}=\left\{x^3y=0\right\}$. Then $A(E)=\frac{3}{5}-\frac{16}{5}b$. So $Fut=0\Rightarrow 45a-7b=24$,
        \item $L^{(2)}_{0}=\left\{xy^2=0\right\}$. Then $A(E)=\frac{3}{5}-\frac{6}{5}b$. So $Fut=0\Rightarrow 15a+b=8$.
    \end{itemize}
    We finished the proof.
\end{proof}

\subsection{The critical $\mathbb{P}(1,4,25)$-pair}

In this case, $X\cong\mathbb{P}(1,4,25)$, $Q^{(3)}\sim\mathcal{O}(50)$, $L^{(3)}\sim\mathcal{O}(10)$. We set the $x$, $y$, $z$ are homogeneous coordinates on $\mathbb{P}(1,4,25)$ with weight $1$, $4$, $25$ respectively.
First of all, we show the following result:

\begin{thm}\label{whenP1425appear}
    If $(\mathbb{P}(1,4,25),aQ^{(3)}+bL^{(3)})$ is K- semistable, then the singular point $[0:0:1]$ of type $\dfrac{1}{25}(1,4)$ is not contained in $Q^{(3)}$. Moreover, in this case the coefficients $(a,b)$ satisfied the constraints $115a+11b\ge63$.
\end{thm}

\begin{proof}
    Take the weighted blow up at the point $[0:0:1]$ with weight $(1,4)$ along the local coordinate $(x,y)$ respectively:
    \[
    \pi: Y\rightarrow X=\mathbb{P}(1,4,25).
    \]
    Denote the exceptional divisor by $E$. Then we have:
    \begin{align*}
        K_{Y}=&\pi^*K_{X}-\frac{4}{5}E;\\
        E^2=&-\frac{25}{4}.
    \end{align*}
    Thus 
    \begin{align*}
        A(E)=&\frac{1}{5}-a\cdot\ord_{E}(Q^{(3)})-b\cdot\ord_{E}(L^{(3)});\\
        S(E)=&\frac{4}{5}(1-\frac{5}{3}a-\frac{1}{3}b).
    \end{align*}
    If $(\mathbb{P}(1,4,25),aQ^{(3)}+bL^{(3)})$ is K- semistable, we obtain:
    \[
    \dfrac{20-15\ord_{E}(Q^{(3)})}{9}a+\dfrac{4-15\ord_{E}(L^{(3)})}{9}b\ge1.
    \]
     Note that $[0:0:1]\in L^{(3)}=\left\{c_1x^2y^2+c_2x^6y+c_3x^{10}=0\right\}$, which implies that $\ord_{E}(L^{(3)})=\dfrac{10}{25}\Rightarrow\dfrac{9}{4-15\ord_{E}(L^{(3})}<0.$
     If $[0:0:1]\in Q^{(3)}$, then the equation of $Q^{(3)}$ is $\left\{zf_{25}(x,y)+f_{50}(x,y)=0\right\}$. Thus $\ord_{E}(Q^{(3)})\ge1\Rightarrow\dfrac{9}{20-15\ord_{E}(Q^{(3)})}\ge\dfrac{9}{5}>\dfrac{3}{5}$. But this is contradicted with $5a+b\le3$, so $[0:0:1]\notin Q^{(3)}$.

     Moreover, consider the point $[0:1:0]$ corresponding to the $\dfrac{1}{4}(1,1)$ singularity in $\mathbb{P}(1,4,25)$. Let $\varphi:\mathbb{A}^2_{x,z}\rightarrow\mathbb{P}(1,4,25)$ be the cyclic quotient map over the open subset $\left\{y=1\right\}\subset\mathbb{P}(1,4,25)$. Set $v:=\varphi_{*}\ord_{0}$, then by computation we know that $A_{\mathbb{P}(1,4,25)}(v)=2$, and 
     \begin{equation*}
         S(v)=\dfrac{52}{15}(3-5a-b).
     \end{equation*}
     Since $(\mathbb{P}(1,4,25),aQ^{(3)}+bL^{(3)})$ is K-semistable, by valuative criterion we have 
     \[
     2-a\cdot v(Q^{(3)})-b\cdot v(L^{(3)})\ge\dfrac{52}{15}(3-5a-b).
     \]
     Since $Q^({3})$ is degree 50, we have $v(Q^({3}))\ge2$ and $v(L^{(3)})\ge2$ because the lowest degree terms of $Q^{(3)}$ and $L^{(3)}$ at $[0:1:0]$ are $x^2y12,xzy^6,z^2$ and $x^2$ respectively. So we deduce that $115a+11b\ge63$.
\end{proof}

From this we can write the equation of $Q^{(3)}$ by $\left\{z^2+zf_{25}(x,y)+f_{50}(x,y)=0\right\}$. Now let's find out all critical pairs associated to $\mathbb{P}(1,4,25)$.

\begin{thm}
    The walls and the corresponding critical pairs associated to $\mathbb{P}(1,4,25)$ are listed in the following table:
    \begin{center}
    \renewcommand*{\arraystretch}{1.2}
    \begin{table}[ht]
    \centering
      \begin{tabular}{ |c |c |c|}
    \hline
     wall &  curve $Q^{(3)}_{0}$ on $\mathbb{P}(1,4,25)$ & curve $L^{(3)}_{0}$ on $\mathbb{P}(1,4,25)$  \\  \hline
      $115a+11b=63$   &  $\left\{z^2+x^2y^{12}=0\right\}$   & $\left\{x^2y^2=0\right\}$ \\ \hline 
     $115a-19b=63$  &  $\left\{z^2+x^2y^{12}=0\right\}$   & $\left\{x^6y=0\right\}$\\ \hline   
     $115a-49b=63$ &     $\left\{z^2+x^2y^{12}=0\right\}$   & $\left\{x^{10}=0\right\}$ \\ \hline  
     $95a+13b=54$ &     $\left\{z^2+x^6y^{11}=0\right\}$   & $\left\{x^2y^2=0\right\}$ \\ \hline
     $95a-17b=54$ &     $\left\{z^2+x^6y^{11}=0\right\}$   & $\left\{x^6y=0\right\}$ \\ \hline
     $95a-47b=54$ &     $\left\{z^2+x^6y^{11}=0\right\}$   & $\left\{x^{10}=0\right\}$ \\ \hline
     \end{tabular}
     \caption{ K-moduli walls from $\mathbb{P}(1,4,25)$}
     \label{Kwall5}
\end{table}
\end{center}
\end{thm}

\begin{proof}
    Let $(\mathbb{P}(1,4,25),aQ^{(3)}+bL^{(3)})$ is a complexity one K-polystable log Fano pair. Assume the torus action is defined by:
    \[
    t\circ[x:y:z]=[t^{n_1}x:t^{n_2}y:t^{n_3}z],\quad \forall t\in\mathbb{C}^*.
    \]
    Note that there is no horizontal divisor under this action.
    Since $Q^{(3)}$ is invariant under this action, we have:
    \begin{align*}
        n_3-25n_1=(6-k)(n_2-4n_1),& \quad k\in\left\{0,1,...,6\right\}\\
        or\quad 2(n_3-25n_1)=(12-l)(n_2-4n_1),&\quad l\in\left\{0,1,...,12\right\}
    \end{align*}
    which corresponding to the equation of $Q^{(3)}$ is 
    \begin{align*}
        \left\{z^2+zf_{25}(x,y)=0\right\}\\
        or\quad \left\{z^2+f_{50}(x,y)=0\right\}
    \end{align*}
    respectively.
    
    Write $m_1=n_2-4n_1$, $m_2=n_3-25n_1$. Take the weighted blow up at $[1:0:0]$ with weight $(m_1,m_2)$ along the coordinate $(y,z)$:
    \[
    \pi:Y\rightarrow\mathbb{P}(1,4,25)=X,
    \]
    Let $E$ be the exceptional divisor. $E^2=-\frac{1}{m_1m_2}$. We have:
    \begin{align*}
        A(E)=&m_1+m_2-a\ord_{E}(Q^{(3)})-b\ord_{E}(L^{(3)});\\
        S(E)=&\frac{25m_1+4m_2}{10}(1-\frac{5}{3}a-\frac{1}{3}b).
    \end{align*}
    Since $Fut(\lambda)=0$, we have $A(E)=S(E)$, thus
    \begin{equation}\label{wall3fut}
    \frac{125m_1-40m_2}{45m_1-18m_2}a+\frac{25m_1+4m_2-30\ord_{E}(L^{(3)})}{45m_1-18m_2}b=1.
    \end{equation}
    On the other hand, $H_x$, $H_y$ are horizontal divisor under above action, and 
    \begin{align*}
        A(H_x)=&1-b\ord_{H_x}(L);\\
        S(H_x)=&10(1-\frac{5}{3}a-\frac{1}{3}b);\\
        A(H_y)=&1-b\ord_{H_y}(L);
        \\
         S(H_x)=&\frac{5}{2}(1-\frac{5}{3}a-\frac{1}{3}b).
    \end{align*}
    Thus 
    \begin{align}\label{wallvert}
        1-b\ord_{H_x}(L)>&10(1-\frac{5}{3}a-\frac{1}{3}b);\\
        1-b\ord_{H_y}(L)>&\frac{5}{2}(1-\frac{5}{3}a-\frac{1}{3}b).\label{vert}
    \end{align}
    \begin{itemize}
        \item If $\ord_{H_{x}}(L^{(3)})=2$, this implies that $\ord_{H_y}(L^{(3)})=2$, $\ord_{E}(L^{(3)})=2m_1$. The inequality(\ref{wallvert}) and (\ref{vert}) can be written as:
        \begin{align}
            25a-7b>&9\label{ineq1}\\
            50a+4b>&27\label{ineq2}
        \end{align}
        Combine with constraints (\ref{ineq1}), (\ref{ineq2}) and (\ref{wall3fut}), we know that it is only possible for $\frac{m_2}{m_1}=6$ or $\frac{11}{2}$ which is corresponding to the invariant curves
        $\left\{z^2+x^2y^{12}=0\right\}$
        or$\left\{z^2+x^6y^{11}=0\right\}$
        respectively. So we obtained the critical pairs $(\mathbb{P}(1,4,25),a\cdot\left\{z^2+x^2y^{12}=0\right\}+b\cdot\left\{x^2y^2=0\right\})$ in which $115a+11b=63$ and $(\mathbb{P}(1,4,25),a\cdot\left\{z^2+x^6y^{11}=0\right\}+b\cdot\left\{x^2y^2=0\right\})$ in which $95a+13b=54$.
        \item If $\ord_{H_{x}}(L^{(3)})=6$, this implies that $\ord_{H_y}(L^{(3)})=1$, $\ord_{E}(L^{(3)})=m_1$. The inequality(\ref{wallvert}) and (\ref{vert}) can be written as:
        \begin{align}
            25a-b>&9\label{ineq3}\\
            50a-8b>&27\label{ineq4}
        \end{align}
         Combine with constraints (\ref{ineq3}), (\ref{ineq4}) and (\ref{wall3fut}), we know that it is only possible for $\frac{m_2}{m_1}=6$ or $\frac{11}{2}$ which is corresponding to the invariant curves
        $\left\{z^2+x^2y^{12}=0\right\}$
        or$\left\{z^2+x^6y^{11}=0\right\}$
        respectively. So we obtained the critical pairs $(\mathbb{P}(1,4,25),a\cdot\left\{z^2+x^2y^{12}=0\right\}+b\cdot\left\{x^6y=0\right\})$ in which $115a-19b=63$ and $(\mathbb{P}(1,4,25),a\cdot\left\{z^2+x^6y^{11}=0\right\}+b\cdot\left\{x^6y=0\right\})$ in which $95a-17b=54$.   
        \item If $\ord_{H_{x}}(L^{(3)})=10$, this implies that $\ord_{H_y}(L^{(3)})=0$, $\ord_{E}(L^{(3)})=0$. The inequality(\ref{wallvert}) and (\ref{vert}) can be written as:
        \begin{align}
            25a+5b>&9\label{ineq5}\\
            50a-20b>&27\label{ineq6}
        \end{align}
         Combine with constraints (\ref{ineq5}), (\ref{ineq6}) and (\ref{wall3fut}), we know that it is only possible for $\frac{m_2}{m_1}=6$ or $\frac{11}{2}$ which is corresponding to the invariant curves
        $\left\{z^2+x^2y^{12}=0\right\}$
        or$\left\{z^2+x^6y^{11}=0\right\}$
        respectively. So we obtained the critical pairs $(\mathbb{P}(1,4,25),a\cdot\left\{z^2+x^2y^{12}=0\right\}+b\cdot\left\{x^{10}=0\right\})$ in which $115a-49b=63$ and $(\mathbb{P}(1,4,25),a\cdot\left\{z^2+x^6y^{11}=0\right\}+b\cdot\left\{x^{10}=0\right\})$ in which $95a-47b=54$.    
    \end{itemize}
    We finished the proof.
\end{proof}

\section{Explicit wall crossings}

So far we have found all of critical lines where the $K$-moduli space $\overline{P}^K_{a,b}$ change. They are attached to surfaces $\mathbb{P}^2$, $\mathbb{P}(1,1,4)$, $\mathbb{P}(1,4,25)$ and $X_{26}$ respectively. The wall-chamber structure of the coefficient domain of $\overline{P}^K_{(a,b)}$ can be draw in a picture as follows( Figure\ref{fig:enter-label}):

\begin{figure}[h!]
    \centering
    \includegraphics[width=0.8\linewidth]{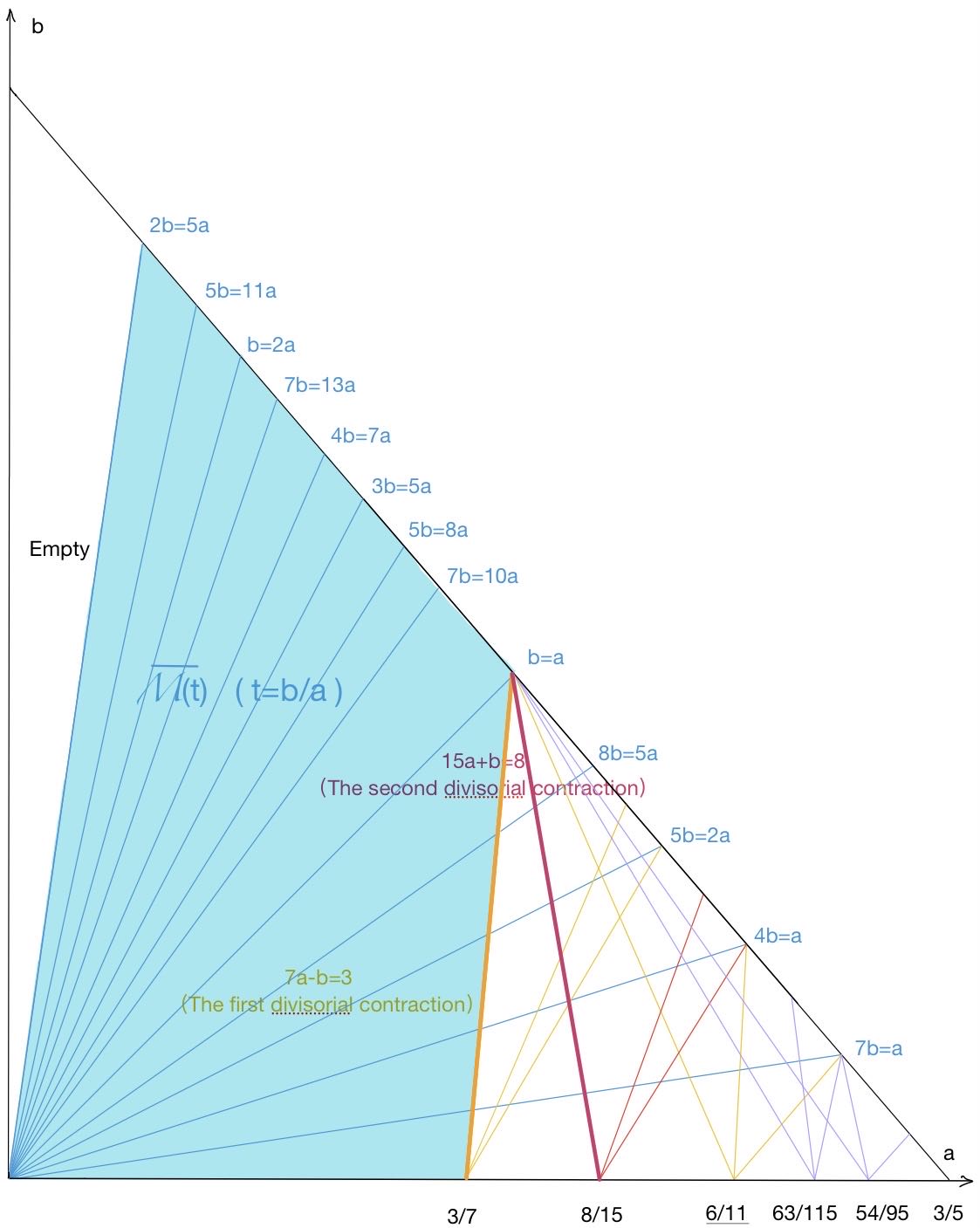}
    \caption{wall-chamber structure of $\overline{P}^K_{a,b}$}
    \label{fig:enter-label}
\end{figure}
The lines are the critical lines that we have computed. More precisely:

There are $13$ lines passing through the origin which corresponding to critical lines associated to $\mathbb{P}^2$. Moreover, the K-moduli space $\overline{P}^K_{a,b}$ is non-empty if and only if $5a\ge2b$ and $5a+b\le3$;

There are $6$ lines with horizontal interception $\frac{3}{7}$ and $\frac{6}{11}$ respectively which corresponding to critical lines associated to $\mathbb{P}(1,1,4)$. Moreover, surface $\mathbb{P}(1,1,4)$ appears only if $7a-b\ge3$;

There are $6$ lines with horizontal interception $\frac{63}{115}$ and $\frac{54}{95}$ respectively which corresponding to critical lines associated to $\mathbb{P}(1,4,25)$ and $\mathbb{P}(1,4,25)$ appears only if $115a+11b\ge63$;

There are also one line with horizontal interception $\frac{8}{15}$ which corresponding to critical lines associated to $X_{26}$ and $X_{26}$ appears only if $15a+b\ge8$.

As is explained in theorem \ref{critical}, the moduli space $\overline{P}^K_{(a,b)}$ does not change as $(a,b)$ varies in domains enclosed by above lines.

In this section we will analyze the what happens on the K-moduli space $\overline{P}^K_{a,b}$ at these critical lines.

\begin{defn}
    We call the critical line is horizontal wall if it arising from the pair $\mathbb{P}^2$. Otherwise, the critical line is said to be vertical wall. More precisely, we call the vertical wall crossing are Type (\uppercase\expandafter{\romannumeral1})-(\uppercase\expandafter{\romannumeral3}) if the general pair parameterized by center is $\mathbb{P}(1,1,4)$, $X_{26}$ or $\mathbb{P}(1,4,25)$ respectively.
\end{defn}

\begin{thm}\label{horizontalwall}
    Let $(a,b)$ be a pair of positive rational numbers such that $5a+b\le3$, $2b\le5a$ and $7a-b\le3$. Then we have the isomorphism between the GIT-moduli space of degree 5 pairs and the K-moduli space: $\overline{\mathcal{M}}(t)\cong\overline{P}_{a,b}^K$, where $t=\dfrac{b}{a}$.
\end{thm}

\begin{proof}
    By Theorem \ref{whenp114appear}, Theorem \ref{whenX26appear} and Theorem \ref{whenP1425appear}, we have known that if the $[(X;Q,L)]\in\overline{P}^K_{a,b}$, then $X\cong\mathbb{P}^2$. Thus by Theorem 1.1 in \cite{zhou2023chamberdecompositionksemistabledomains}, we known that there is an isomorphism between the moduli space: $\overline{\mathcal{M}}(\frac{b}{a})\cong\overline{P}_{a,b}^K$.
\end{proof}

From the above theorem, we conclude that all horizontal wall crossing morphism of the K-moduli space $\overline{P}^K_{a,b}$ can be identify to the VGIT wall crossing morphism described as in \cite{laza2007deformationssingularitiesvariationgit}. Moreover, we can always deduced from above theorem that $\overline{P}^{K}_{1,1}\cong(\mathcal{D}/\Gamma)^*$.

Next let's focus on the vertical horizontal wall crossings. Let $(a_i,b_i)$ be a pair of positive rational numbers lies on the vertical wall and let $\Sigma_{a_i,b_i}$ be the closure of the union of the points on $\overline{P}^K_{a_i,b_i}$ that corresponding critical pairs. We call it the center of $\overline{P}^K_{a,b}$ on this wall. We know that the wall crossing morphisms 
\[
\phi^{\pm}_{a_i\pm\varepsilon,b_i\pm\eta}:\overline{P}^K_{a_i\pm\varepsilon,b_i\pm\eta}\rightarrow\overline{P}^K_{a_i,b_i}
\] 
are birational morphism and it is isomorphism over $\overline{P}^K_{a_i,b_i}\backslash\Sigma_{a_i,b_i}$. So to describe the vertical wall crossing of the K-moduli space we need to investigate its exceptional loci. 

\begin{defn}
    We call the preimage of center under the wall crossing morphism $\phi^{\pm}_{a_i,b_i}$ by exceptional locus and denote it by $E^{\pm}_{a_i,b_i}$. Equivalently,
    \[
    \left\{
    \begin{array}{cc} \text{$[(X;Q,L)]\in\overline{P}^K_{a_i\pm\varepsilon,b_i\pm\eta}$ and 
    admits a special degeneration }&\\
    \text{to K-polystable log Fano pair parameterized by $\Sigma_{a_i,b_i}$.}&\\
    \end{array}\right\}
    \]
\end{defn}
  
\subsection{The Type(\uppercase\expandafter{\romannumeral1}) vertical wall crossing}

\begin{notation}
    Let $C$ be a smooth conic in $\mathbb{P}^2$, let $L$ be a line in $\mathbb{P}^2$ transverse to $C$, and denote $Q_1=2C+L$. Let $Q^{(2)}$ and $L^{(1)}$ be the degree 10 curve and degree 2 curve on $\mathbb{P}(1,1,4)$ which defined by $\left\{z^2xy=0\right\}$ and $\left\{xy=0\right\}$ respectively.
\end{notation}

\begin{thm}\label{P2toP114}
    The log Fano pair $(\mathbb{P}^2,a_1Q_1+b_1L)$ is K-semistable with K-polystable degeneration $(\mathbb{P}(1,1,4),a_1Q^{(1)}+b_1L^{(1)})$. Moreover, if $L'$ is another line in $\mathbb{P}^2$ which is distinct to $L$, then  $(\mathbb{P}^2,a_1Q_1+b_1L')$ is also K-semistable and admits a special degeneration to K-polystable pair $(\mathbb{P}(1,1,4),a_1Q^{(1)}+b_1L^{(1)'})$, where $L^{(1)'}$ is the union of two lines on $\mathbb{P}(1,1,4)$ and both of them are distinct with $\{x=0\}$ and $\{y=0\}$.
\end{thm}

\begin{proof}
    Taking the degeneration of $\mathbb{P}^2$ to the normal cone of $C$, the pair $(\mathbb{P}^2,C)$ specially degenerates to $(\mathbb{P}(1,1,4),\{z=0\})$. Since $L$ intersect $C$ transversely, so it degenerates to union of two distinct rulings of $\mathbb{P}(1,1,4)$. After a suitable coordinate change, we may assume that the degeneration of $L$ is defined by $\{xy=0\}$. The degeneration of $L'$ can be written as $\{l_11l_2=0\}$ where $l_i\neq\{x\}$ and $\{y=0\}$. Note that if $L'$ is tangent to $C$, then $l_1=l_2$. Thus we shown that $(\mathbb{P}^2,a_1Q_1+b_1L)$ and $(\mathbb{P}^2,a_1Q_1+b_1L')$ can specially degenerate to $(\mathbb{P}(1,1,4),a_1Q^{(1)}+b_1L^{(1)})$ and $(\mathbb{P}(1,1,4),a_1Q^{(1)}+b_1L^{(1)'})$ respectively. The K-polystability of $(\mathbb{P}(1,1,4),a_1Q^{(1)}+b_1L^{(1)})$ and $(\mathbb{P}(1,1,4),a_1Q^{(1)}+b_1L^{(1)'})$ has been proven in the Theorem \ref{4.8}.
\end{proof}

In order to describe the exceptional locus of the wall crossing morphism, we define GIT stability for certain pairs of curves on $\mathbb{P}(1,1,4)$. 

\begin{defn}\label{GITonP114}
Denote the $x,y,z$ be coordinates of $\mathbb{P}(1,1,4)$.
\begin{enumerate}
    \item Let $Q$ be a curve in $\mathbb{P}(1,1,4)$ of degree $10$ with equation 
\begin{equation*}
    z^2xy+zf_{6}(x,y)+f_{10}(x,y)=0
\end{equation*}
where $f_6(x,y)$ contains no monomial divisible by $xy$. Then we identify $Q$ to a point $(f_6,f_{10})$ in the vector space 
\[
\mathbf{A}'_{10}:=V_1\oplus H^{0}(\mathbb{P}^1,\mathcal{O}_{\mathbb{P}^1}(10)),
\]
where $V_1:=\mathbb{C}x^6\oplus\mathbb{C}y^6$ is a subvector space of $H^0(\mathbb{P}^1,\mathcal{O}_{\mathbb{P}^1}(6))$. Consider the $\mathbb{G}_m$-action $\sigma$ on $\mathbf{A}'_{10}$ with weight $1$ on $V_1$ and weight $2$ on $H^{0}(\mathbb{P}^1,\mathcal{O}_{\mathbb{P}^1}(10))$. Let $\mathbf{P}'_{10}$ be the weight projective space which is the coarse moduli space of the quotient stack $[(\mathbf{A}'_{10})\backslash\left\{0\right\}/\mathbb{G}_m]$.
\item Let $L$ be a curve in $\mathbb{P}(1,1,4)$ of degree $2$ with equation
\[
\lambda_1x^2+\lambda_2xy+\lambda_3y^2=0
\]
Then we can identify $L$ to a point in $\mathbf{A}'_{2}=H^0(\mathbb{P}^1,\mathcal{O}_{\mathbb{P}^2}(2))$. Let $\mathbf{P}'_{2}$ be the corresponding projective space.
\end{enumerate}
We Consider a $\mathbb{G}_m$-action $\sigma'$ on $\mathbf{A}'_{10}\times\mathbf{A}'_{2}$ induced by
\begin{equation*}
    \sigma'(t)[x,y]=[tx,t^{-1}y].
\end{equation*}
on $\mathbb{P}^1$. It is clearly that it descends to a $\mathbb{G}_m$-action on $(\mathbf{P}'_{10}\times\mathbf{P}'_{2},\mathcal{O}_{\mathbf{P}'_{10}\times\mathbf{P}'_{2}}(a,b))$ which is also denote by $\sigma'$. We say $([Q],[L])\in\mathbf{P}'_{10}\times\mathbf{P}'_{2}$ is GIT (poly/semi)stable with slope $(a,b)$ if it is GIT (poly/semi)stable with respect to the $\mathbb{G}_m$-action $\sigma'$ on $(\mathbf{P}'_{10}\times\mathbf{P}'_{2},\mathcal{O}_{\mathbf{P}'_{10}\times\mathbf{P}'_{2}}(a,b))$.
\end{defn}

\begin{thm}
    Let $(a_1,b_1)$ be a pair of positive rational numbers satisfied $7a_1-b_1=3$ and $5a_1+b_1\le3$, let $(\varepsilon,\eta)$ be a pair of rational numbers satisfied following constraints:
    \begin{equation}\label{epsilon}
        \left\{(\epsilon,\eta)|\text{$7\epsilon-\eta>3$, $35(a_1+\epsilon)+17(b_1+\eta)\le15$, $15(a_1+\epsilon)+(b_1+\epsilon)\le8$}\right\}
    \end{equation}
    Then for any $(Q,L)\in\mathbf{A}'_{10}\times\mathbf{A}'_{2}$, the pair $(\mathbb{P}(1,1,4),(a_1+\varepsilon)Q+(b_1+\eta)L)$ is K-(semi/poly)stable if and only if $([Q],[L])$ is GIT-(semi/poly)stable with slope $(a_1+\epsilon,b_1+\eta)$.
\end{thm}

\begin{proof}
    Let $\pi:(\mathbb{P}(1,1,4)\times\mathbf{A}'_{10}\times\mathbf{A}'_{2};\mathcal{Q},\mathcal{L})\rightarrow\mathbf{A}'_{10}\times\mathbf{A}'_{2}$ be the universal family of pairs over $\mathbf{A}'_{10}\times\mathbf{A}'_{2}$ where the fiber of $\pi$ over each point of $(Q,L)$ of $\mathbf{A}'_{10}\times\mathbf{A}'_{2}$ is $(\mathbb{P}(1,1,4);Q,L)$. The $\mathbb{G}_m$-action $\sigma$ on $\mathbf{A}'_{10}\times\mathbf{A}'_{2}$ can be lifted to the universal family, which we also denote by $\sigma$. By quotienting out $\sigma$ we obtain a $\mathbb{Q}$-Gorenstein family of log Fano pairs over the Deligne-Mumford stack $[(\mathbf{A}'_{10}\backslash\left\{0\right\})/\mathbb{G}_m]\times\mathbf{P}'_{2}$. The CM $\mathbb{Q}$-line bundle $\lambda_{CM,\pi,a\mathcal{Q}+b\mathcal{L}}$ on $\mathbf{A}'_{10}\times\mathbf{A}'_{2}$ also descends to a $\mathbb{Q}$-line bundle on $\mathbf{P}'_{10}\times\mathbf{P}'_{2}$ which we denote by $\Lambda_{a,b}$. In order to show that the K-stability implies GIT-stability, we only need to show that $\Lambda_{a_1+\epsilon,b_1+\eta}$ is ample. Note that $\lambda_{CM,\pi,a\mathcal{Q}+b\mathcal{L}}$ is a trivial $\mathbb{Q}$-line bundle on $\mathbf{A}'_{10}\times\mathbf{A}'_{2}$, the degree of $\Lambda_{a,b}$ is equal to the $\sigma$-weight of the central fiber $\lambda_{CM,\pi,a\mathcal{Q}+b\mathcal{L}}\otimes\mathbb{C}(0)$. We know that 
    \[
    \deg\Lambda_{a,b}=Fut((\mathbb{P}(1,1,4),aQ^{(1)}_1+bL^{(1)}_1;\mathcal{O}_{\mathbb{P}(1,1,4)}(4))\times\mathbb{A}^1)
    \]
    where the product test configuration $(\mathbb{P}(1,1,4),aQ^{(1)}_1+bL^{(1)}_1;\mathcal{O}_{\mathbb{P}(1,1,4)}(4))\times\mathbb{A}^1$ is induced by the $\mathbb{G}_m$-action $\sigma$. $Fut((\mathbb{P}(1,1,4),aQ^{(1)}_1+bL^{(1)}_1;\mathcal{O}_{\mathbb{P}(1,1,4)}(4))\times\mathbb{A}^1)$ is linear in the coefficient $(a,b)$, and we know that it is negative when $(a,b)=(0,0)$, and zero when $(a,b)=(a_1,b_1)$. Hence it is positive when $(a,b)=(a_1+\epsilon,b_1+\eta)$. As a result, the CM line bundle $\Lambda_{a_1+\epsilon,b_{1}+\eta}$ is ample on $\mathbf{P}'_{10}\times\mathbf{P}'_{2}$.

    Next we show that the GIT-stability implies K-stability. Fixed any $(\varepsilon,\eta)$ satisfied the conditions (\ref{epsilon}). For any pair $([Q],[L])$ in the $\mathbf{P}'_{10}\times\mathbf{P}'_{2}$, if it is GIT-semistable with slope $(a_1+\varepsilon,b_1+\eta)$ then $L$ is not equals to $\left\{x^2=0\right\}$ or $\left\{y^2\right\}=0$. Assume that $L=\left\{l_1l_2=0\right\}$. Note that $(\mathbb{P}(1,1,4);Q,L)$ admits a special degeneration to the $(a_1,b_1)$-K-semistable pair $(\mathbb{P}(1,1,4);Q^{(1)},L)$ via the $\mathbb{C}^*$-action $t\circ[x:y:z]=[x:y:tz]$, so these pairs are both $(a_1,b_1)$-$K$-semistable. On the other hand, take $\left\{Q_{t}\right\}$ a family of degree $10$ curves on $\mathbb{P}(1,1,4)$ over a smooth pointed curve $(0\in T)$ such that $Q_1=Q$ and $(\mathbb{P}(1,1,4),(a_1+\varepsilon)Q_t+(b_1+\eta)L)$ is K-semistable for $t\in T\backslash\left\{0\right\}$. By properness of K-moduli spaces we have a K-polystable limit $(X,(a_1+\varepsilon)\widetilde{Q}+(b_1+\eta)\widetilde{L})$ of $(\mathbb{P}(1,1,4),(a_1+\varepsilon)Q_t+(b_1+\eta)L)$ as $t$ goes to $0$. By previous analyse we know that $X$ can only possibly be $\mathbb{P}^2$ or $\mathbb{P}(1,1,4)$. But if $X\cong\mathbb{P}^2$, then by interpolation of K-stability we know that $(\mathbb{P}^2,a_1\widetilde{Q}+b_1\widetilde{L}))$ is also K-polystable. However, this implies that $(\mathbb{P}^2,a_1\widetilde{Q}+b_1\widetilde{L}))\cong(\mathbb{P}(1,1,4),a_1Q^{(1)}_1+b_1L^{(1)}_{1})$ which is a contradiction! Thus $X\cong\mathbb{P}(1,1,4)$ and $(\widetilde{Q},\widetilde{L})$ is a pair of curves on $\mathbb{P}(1,1,4)$ which is $GIT$-polystable with slope $(a_1+\varepsilon,b_1+\eta)$ since we have shown that the K-polystability implies GIT-semistability as above. By the properness of GIT, we know that $(\mathbb{P}(1,1,4),(a_1+\varepsilon)Q+(b_1+\eta)L)$ admits a special degeneration to $(\mathbb{P}(1,1,4),(a_1+\varepsilon)\widetilde{Q}+(b_1+\eta)\widetilde{L})$ which implies that $(\mathbb{P}(1,1,4),(a_1+\varepsilon)Q+(b_1+\eta)L)$ is K-semistable. If $(Q,L)$ is GIT-polystable, then $(Q,L)\cong(\widetilde{Q},\widetilde{L})$ which implies that $(\mathbb{P}(1,1,4),(a_1+\varepsilon)Q+(b_1+\eta)L)\cong(\mathbb{P}(1,1,4),(a_1+\varepsilon)\widetilde{Q}+(b_1+\eta)\widetilde{L})$ is also K-polystable. Thus we prove the statement.
\end{proof}

Now we can describe the wall crossing morphism at the first vertical wall more precisely:

\begin{thm}\label{type1wallcrossing}
    Let $(a_1,b_1)$ be a pair of positive rational numbers such that $7a_1-b_1=3$ and $5a_1+b_1\le3$. Fix any another two pairs of rational numbers $(a_{1}^{-},b_{1}^{-})$ and $(a_{1}^{+},b_{1}^{+})$ such that $7a_{1}^{-}-b_{1}^{-}<3$ and $7a_{1}^{+}-b_{1}^{+}>3$. Moreover ,we assume that they are lying in the two distinct chambers which has common face along the the critical line $7a-b=3$. 
    \begin{enumerate}
        \item The general point of the center $\Sigma_{a_1,b_1}$ parameterize the pair $(\mathbb{P}(1,1,4);Q^{(1)},L^{(1)})$, where $L^{(1)}$ is defined by $\left\{xy=0\right\}$ or $\left\{\text{$l_1l_2=0$|$l_i\neq x$ and $y$}\right\}$. In particular, $\dim\Sigma_{a_1,b_1}=1$.
        \item The wall crossing morphism $\phi^{-}_{a_1,b_1}:\overline{P}^K_{a_{1}^{-},b_{1}^{-}}\rightarrow\overline{P}^K_{a_1,b_1}$ is isomorphism which take the points $[(\mathbb{P};Q_{0},L)]$ to $[(\mathbb{P}(1,1,4);Q^{(1)},\{xy=0\})]$ and take the point $[(\mathbb{P};Q_{0},L')]$ to $[(\mathbb{P}(1,1,4);Q^{(1)},\{L^{(2)'}=0\})]$.
        \item The wall crossing morphism $\phi^{+}_{a_1,b_1}:\overline{P}^K_{a_1^+,b_1^+}\rightarrow\overline{P}^K_{a_1,b_1}$ is a divisorial contraction. The exceptional divisor $E^{+}_{a_1,b_1}$ parameterize pair $(\mathbb{P}(1,1,4);Q,L)$ such that $(Q,L)$ is GIT-polystable with slope $(a_1^+,b_1^+)$ in the sense of definition \ref{GITonP114}.
    \end{enumerate}
\end{thm}

\begin{proof}
    For the description about the wall crossing center, we have shown in the Theorem \ref{4.8}. It is easy to see that $\dim\Sigma_{a_1,b_1}=1$. By Theorem \ref{P2toP114}, we have known that $\phi^{-}_{a_1,b_1}$ replace $(\mathbb{P}^2;Q,L)$ and $(\mathbb{P}^2;Q,L')$ by $(\mathbb{P}(1,1,4);Q^{(1)},L^{(1)})$ and $(\mathbb{P}(1,1,4);Q^{(1)},L^{(2)'})$ respectively. We know that the K-moduli space $\overline{P}^K_{a,b}$ is normal and $\phi^{-}_{a_1,b_1}$ is isomorphism over $\overline{P}^K_{a_1,b_1}\backslash\Sigma_{a_1,b_1}$. Hence to show $\phi^{-}_{a_1,b_1}$ is isomorphism, we only need to show that the fiber of $\phi^{-}_{a_1,b_1}$ over any point of $\Sigma_{a_1,b_1}$ has dimension 0. Applying \cite{VGIT} or \cite{GITandFlips} to the local VGIT presentation of K-moduli space implies that
    \[
    d^{-}+d^{+}+1=\codim(\Sigma_{a_1,b_1})=13.
    \]
    where $d^{\pm}$ is the dimension of the fiber of $E^{\pm}_{a_1,b_1}$ over $\Sigma$ respectively.
    If $[(X;D,H)]\in E^{+}_{a_1,b_1}$, by index estimate we know that the only possible of $X$ is $\mathbb{P}^2$, $\mathbb{P}(1,1,4)$, $X_{26}$ or $\mathbb{P}(1,4,25)$. But we have shown in Theorem \ref{whenX26appear} and Theorem \ref{whenP1425appear} that the $X_{26}$-pair can only appear in moduli space when $15a+b\ge8$ and $\mathbb{P}(1,4,25)$-pair appear only if $115a+11b\ge63$. If $X\cong\mathbb{P}^{2}$, by the interpolation of the K-stability, we know that $(\mathbb{P}^2,a_1D+b_1H)$ is K-polystable. But by definition it admit a special degeneration to a K-polystable $\mathbb{P}(1,1,4)$-pair on center, thus they are isomorphism which is a contradiction! Thus we have shown that the only possible case is $X\cong\mathbb{P}(1,1,4)$. In this case, $D$ is degree 10 curve on $\mathbb{P}(1,1,4)$ and by Theorem\ref{whenp114appear}, we can assume that the equation of $D$ is
    \[
    z^2xy+zf_6(x,y)+f_{10}(x,y)=0
    \]
    where $f_6$ has no term divisible by $xy$. Combine with Theorem \ref{epsilon}, we can deduce that the fiber of $E^{+}_{a_1,b_1}$ over the point $[(\mathbb{P}(1,1,4),\{z^2xy=0\},\{xy=0\})]\in\Sigma_{a_1,b_1}$ is isomorphism to the GIT quotient $\mathbf{P}'_{10}\times\{xy=0\}\q_{\mathcal{O}_{\mathbf{P}'_{10}\times\mathbf{P}'_{2}}(a_1^+,b_1^+)}\mathbb{G}_m$, which dimension is $12$. This finishes the proof.
\end{proof}

\subsection{The Type(\uppercase\expandafter{\romannumeral2}) vertical wall crossing}

\begin{notation}
    Let $Q_2$ be the plane curve defined by
\[
\left\{(y^2-xz)^2(\frac{1}{4}x+y+z)-x^2(y^2-xz)(x+2y)+x^5=0\right\}.
\]
It is a plane quintic curve with a $A_{12}$-singularity. Take $6$-jet $x'=x-y^2+y^5-\dfrac{1}{2}y^6$, then in the coordinates $(x',y)$ the equation of $Q_2$ becomes
\[
x^{'2}=ay^{13}+\text{higher order terms, where $a\neq0$.}
\]
Let $L_2$ be the line defined by $\{x=0\}$.
\end{notation} 

\begin{thm}
    The log Fano pair $(\mathbb{P}^2,aQ_2+bL_2)$ is K-semistable if and only if $15a+b\le8$. Moreover, it admits a special degeneration to K-polystable pair $(X_{26},a_2Q^{(2)}+b_2L^{(2)})$.
\end{thm}

\begin{proof}
    Firstly, let's show that if $(\mathbb{P}^2,aQ_2+bL_2)$ is K-semistable then $15a+b\le8$. Let $\pi: X\rightarrow\mathbb{P}^2$ be the weighted blow up morphism along the coordinate $(x',y)$ with weight $(13,2)$. Denote the exceptional divisor by $E$. Straightforward computation shows that 
    \[
    \text{$A(E)=15-26a-4b$ and $S(E)=\dfrac{51}{5}(1-\dfrac{5}{3}a-\dfrac{1}{3}b)$}.
    \]
    Since $(\mathbb{P}^2,aQ_2+bL_2)$ is K-semistable, by the valuative criterion we know that $A(E)\ge S(E)$, which implies that $15a+b\le8$.

    Next we construct the special degeneration. This construction is the same as in \cite{ADL19}:
    \[
    \begin{tikzcd}[column sep=small]
    & \mathcal{X} \arrow[dl,"\pi"]\arrow[rr, dashrightarrow,"f"] \arrow[dr,"g"] & &\mathcal{X}^{+}\arrow[dr,"\psi"]\arrow[dl,"\psi"]\\
    \mathbb{P}^2\times\mathbb{A}^1 & & \mathcal{Y}&&\mathcal{Z}
    \end{tikzcd}
    \]
    where the $\pi$ is the $(13,2,1)$-weighted blow up of $\mathbb{P}^2\times\mathbb{A}^1$ in the coordinates $(x',y,t)$ where $t$ is the parameter of $\mathbb{A}^1$, the map $g$ is the contraction of $\overline{Q}_{2}$ in $X\subset\mathcal{X}_0$ where $\overline{Q}_2$ is the strict transform of $Q_2$ in $X$, and $\psi$ is the divisorial contraction that contracts $X'$ to a point. The central fiber is $X_{26}$ where $Q_2$ specially degenerates to $\{w=0\}$ and $L_2$ specially degenerates to $\{xy^2=0\}$. 
\end{proof}

\begin{thm}\label{type2wallcrossing}
    Let $(a_2,b_2)$ be a pair of positive rational numbers such that $15a_2+b_2=8$ and $5a_1+b_1\le3$. Fix any another two pairs of rational numbers $(a_{1}^{-},b_{1}^{-})$ and $(a_{1}^{+},b_{1}^{+})$ such that $15a_{1}^{-}+b_{1}^{-}<8$ and $15a_{1}^{+}+b_{1}^{+}>8$. Moreover ,we assume that they are lying in the two distinct chambers which has common face along the the critical line $15a+b=8$.
    \begin{enumerate}
        \item The wall crossing morphism $\phi^{-}_{a_2,b_2}:\overline{P}^K_{a_{2}^{-},b_{2}^{-}}\rightarrow\overline{P}^K_{a_2,b_2}$ is isomorphism which take the points $[(\mathbb{P};Q_{2},L)]$ to $[(X_{26};Q^{(2)},\{xy^2=0\})]$.
        \item The wall crossing morphism $\phi^{+}_{a_2,b_2}:\overline{P}^K_{a_{2}^{+},b_{2}^{+}}\rightarrow\overline{P}^K_{a_2,b_2}$ is a divisorial contraction. The exceptional divisor $E^{+}_{a_2,b_2}$ parametrizes the pair of curves on $X_{26}$ of the form $(\left\{w=g(x,y)\right\},xy^2=h(x,y))$ where $g\neq0$ and $g$ does not contain the term $xy^{12}$, $h$ is degree $5$ of the form $\left\{\lambda x^5+\mu x^3y\right\}$.
    \end{enumerate}
\end{thm}

\begin{proof}
    The proof is similar to that of Theorem 7.1 in \cite{ADL19}.
\end{proof}

\vspace{0.5cm}
\bibliographystyle{alpha}
\bibliography{main}
\end{document}